\documentclass[11pt]{amsart}
\usepackage[all,tips]{xy}
\usepackage{latexsym, amsfonts, amsmath, amssymb, mathrsfs, graphicx, xcolor, hyperref, float, xurl,tikz,tikz-cd,caption}
\usepackage{makecell}
\usepackage{mathtools}
\allowdisplaybreaks

\newtheorem{theorem}{Theorem}[section]
\newtheorem{lemma}[theorem]{Lemma}
\newtheorem{example}[theorem]{Example}
\newtheorem{remark}[theorem]{Remark}
\newtheorem{lemma-definition}[theorem]{Lemma-Definition}
\newtheorem{corollary}[theorem]{Corollary}
\newtheorem{proposition}[theorem]{Proposition}
\newtheorem{definition}[theorem]{Definition}

\newtheorem{thmx}{\text{Theorem}}

\DeclareMathOperator{\End}{End}

\DeclareMathOperator{\trd}{trd}
\DeclareMathOperator{\nrd}{nrd}
\DeclareMathOperator{\Aut}{Aut}
\DeclareMathOperator{\Frob}{Frob}

\DeclareMathOperator{\Spec}{Spec}

\newcommand{\bC}{\mathbb{C}}

\newcommand{\bF}{\mathbb{F}}

\newcommand{\bQ}{\mathbb{Q}}

\newcommand{\bZ}{\mathbb{Z}}

\newcommand{\sC}{\mathscr{C}}

\newcommand{\cD}{\mathcal{D}}

\newcommand{\cO}{\mathcal{O}}

\usepackage[margin=1in]{geometry}
\begin{document}
\title{Atkin--Lehner Quotients of Modular Curves at Primes of Bad Reduction}
\author{Lea Beneish}
\address{GAB 417A; 225 Avenue E, Denton, TX 76201
}
\email{lea.beneish@unt.edu}
\author{Boya Wen}
\address{221 Richmond Way
Jepson Hall, MATH
Richmond, VA 23173
}
\email{boya.wen@richmond.edu}
\begin{abstract}
   Let $N$ be a square-free, positive integer and let $p\ge 5$ be a prime dividing $N$. In this work, we construct semistable models of Atkin--Lehner quotients of $X_0(N)$ over the Witt ring $W(\overline{\bF}_p)$ by taking the corresponding quotients of the Deligne--Rapoport model of $X_0(N)$. We 
   give an explicit description of the special fiber of the models and obtain genus formulae for the Atkin--Lehner quotients of $X_0(N)$.
\end{abstract}
\maketitle
\tableofcontents

\section{Introduction}
Let $N$ be a positive integer. The modular curve $X_0(N)$ is the compactification of the moduli space for degree-$N$ cyclic isogenies between elliptic curves (or equivalently, elliptic curves with marked cyclic subgroups of order $N$). Let $D$ be a positive divisor of $N$ such that $(D,N/D)=1$. The Atkin--Lehner involution $w_D$ acts on $X_0(N)$ as follows from the moduli perspective: Given a degree-$N$ cyclic isogeny $\phi:E\to E'$ between (generalized) elliptic curves, $w_D([\phi])$ is the isomorphism class of the composition isogeny $E/C_D\to E/\ker(\phi)\cong E'\to E'/C_D'$, where $C_D$ and $C_D'$ respectively denote the unique order-$D$ subgroups of $\ker(\phi)$ and $\ker(\hat{\phi})$, $E/C_D\to E/\ker(\phi)$ and $E'\to E'/C_D'$ are natural projections, and $E/\ker(\phi)\cong E'$ is induced by $\phi:E\to E'$. In particular, $w_N([\phi])=[\hat{\phi}]$, the isomorphism class of the dual isogeny $\hat{\phi}:E'\to E$.  (See for example \cite[I.1, I.2]{Gross--Zagier}.) Equivalently, given a (generalized) elliptic curve $E$ with a marked cyclic subgroup $C_N$, $w_D$ sends the point $[(E, C_N)]$ to $[(E/C_D, (C_N+E[D])/C_D)]$, where $C_D$ is the unique order-$D$ subgroup of $C_N$. (See for example \cite[p.1]{Xue}.)

Composition of Atkin--Lehner operators satisfy $w_{D_1}w_{D_2}=w_{D_3}$ where $D_3=D_1D_2/\gcd(D_1,D_2)^2$, (which we shall denote $D_1 * D_2$ in Definition~\ref{def:star}). The group $W$ generated by all the Atkin--Lehner operators $w_D$, where $D$ is a positive divisor of $N$ coprime to $N/D$, is an abelian group isomorphic to $(\bZ/2\bZ)^r$, where $r$ is the number of distinct prime factors of $N$. In this work, we study the quotients of $X_0(N)$ by any subgroup of $W$ (including $W$ itself), which we refer to as \emph{Atkin--Lehner quotients} of $X_0(N)$, under the assumption that $N$ is square-free. 

Atkin--Lehner quotients of $X_0(N)$, particularly $X_0^+(N):=X_0(N)/\langle w_N\rangle$ and $X_0^*(N):=X_0(N)/W$, are also of significant moduli theoretic interest. For example, every non-CM $\mathbb{Q}$-curve (respectively $K$-curve) is isogenous to one whose isomorphism class is parametrized by a $\bQ$-rational (respectively $K$-rational) point of $X_0^*(N)$ \cite{Elkies,Ellenberg}.

Let $k$ be a complete local field with ring of integers $\mathcal{O}_k$ and residue field $\kappa$. Let $C$ be smooth projective and geometrically connected curve defined over $k$. We say $C$ has a model $\sC$ over $\cO_k$ if there is a proper flat scheme morphism $\sC \to \Spec\mathcal{O}_k$ together with a $k$-scheme isomorphism $\sC_{\eta}\cong C$. When a curve does not have a model with good reduction over $\mathcal{O}_k$ one asks for a semistable or stable model (see \cite[Chapter 10, Definition 3.1, 3.2, and 3.14]{Liu_ArithmeticCurve}). Note that the literature sometimes differs in its convention on whether a semistable model is required to be regular. In this paper, we say $\sC$ is a semistable model if the geometric special fiber $\sC_{s,\overline{\kappa}}$ (extending the scalars of the special fiber $\sC_s$ to $\overline{\kappa}$) is reduced and any singularities on $\sC_{s,\overline{\kappa}}$ are ordinary double points (nodes). We do not require the ordinary double points to have thickness $1$, or in other words, we do not require a semistable model $\sC$ to be regular. (See also Lemma-Definition~\ref{lemma-definition} and Remark~\ref{rmk:may-not-regular}.)

Stable and semistable models for various levels of modular curves have been constructed by, for example, \cite{Deligne--Rapoport} ($X_0(Mp)$, semistable, further developed in \cite{Katz--Mazur}), \cite{Edixhoven} ($X_0(Mp^2)$, stable), \cite{Bouw--Wewers} ($X_0(p)$, stable), \cite{Coleman--McMurdy} ($X_0(p^3)$, stable), \cite{Weinstein} ($X(Mp^m)$, semistable), and \cite{Tsushima} ($X_0(p^4)$, stable), where $p$ denotes the characteristic of the residue field of the base discrete valuation ring for the respective models, and $M$ is coprime to $p$.

Given a semistable model $\sC$, one can define the dual graph of the geometric special fiber $\sC_{s,\bar{\kappa}}$ via incidence relations i.e., each irreducible component of $\sC_{s,\bar{\kappa}}$  corresponds to a vertex and each intersection between two irreducible components corresponds to an edge connecting the corresponding vertices.

Our work concerns explicit special fiber (or equivalently, dual graph) descriptions of semistable models of Atkin--Lehner quotients of modular curves. Let $N$ be a square-free positive integer divisible by a prime $p\ge 5$. Take the Deligne--Rapoport model of $X_0(N)$ over the Witt ring $W(\overline{\mathbb{F}}_p)$, whose special fiber $X_0(N)_{\overline{\mathbb{F}}_p}$ consists of two irreducible components of the smooth curve $X_0(N/p)_{\overline{\mathbb{F}}_p}$ intersecting transversally at supersingular points \cite[Section VI Th\'eor\`eme 6.9]{Deligne--Rapoport}. We track the changes in singularity as we quotient by the Atkin--Lehner operators, obtaining semistable models of Atkin--Lehner quotients. We give explicit description of the dual graph of the special fiber and obtain explicit genus formulae for the quotients. While the semistable quotient model we obtain by taking the quotient of the Deligne--Rapoport model need not be regular, in \cite{Xue}, Xue constructs minimal resolution for the one-layer quotient  $X_0(N)/\langle w_M\rangle$.

We have the following theorem for square-free $N$:
\begin{thmx} [Theorem~\ref{thm:generator-free} in the main text]
    Let $N=p_1\dots p_r$ be a square-free positive integer, let $p=p_r$, and let $H$ be a subgroup of $W_N$. Let $\cD_H:=\{d : w_d\in H\}$.
    \begin{enumerate}
        \item If $p\nmid d$ for any $d\in \cD_H$, then $X_0(N)_{\overline{\bF}_p}/H$ obtained from taking the quotient of the Deligne--Rapoport model of $X_0(N)_{\overline{\bF}_p}$ is two irreducible components $X_0(N/p)_{\overline{\bF}_p}/H$ intersecting at 
        $$\frac{1}{|H|}\sum_{d\in \cD_H} F(d)$$
        ordinary double points. Consequently, $$g(X_0(N)/H)=2g(X_0(N/p)/H)+\left(\frac{1}{|H|}\sum_{d\in \cD_H} F(d)\right)-1.$$

        \item Otherwise, there exists some $w_{d_0}\in H$ such that $p\mid d_0$ and $K:=\{w_d\in H : p\nmid d\}$ is an index-$2$ subgroup of $H$. In this case, $X_0(N)_{\overline{\bF}_p}/H=(X_0(N)_{\overline{\bF}_p}/K)/\langle w_{d_0} \rangle$ obtained from taking the quotient of the Deligne--Rapoport model of $X_0(N)_{\overline{\bF}_p}$ is one irreducible component $X_0(N/p)_{\overline{\bF}_p}/K$ with
        $$\frac{1}{|H|}\sum_{d\in \cD_K} \left(F(d)-F(d*d_0)\right)$$
        self-crossings that are ordinary double points, and consequently, $$g(X_0(N)/H)=g(X_0(N/p)/K)+\frac{1}{|H|}\sum_{d\in \cD_K} \left(F(d)-F(d*d_0)\right).$$
    \end{enumerate}  
\end{thmx}

Here and throughout, $d*d_0:=dd_0/\gcd(d,d_0)^2$, $F(D)$ is the shorthand notation for $F(D,N/D,p)$, the number of $w_d$-fixed supersingular points on $X_0(N)_{\overline{\bF}_p}$. We give explicit formulae for the number $F(D,m,p)$ of $w_D$-fixed supersingular points on $X_0(Dm)_{\overline{\bF}_p}$ in Proposition~\ref{prop:fixed-point-formula}.

We note that there are well-known genus formulae for quotients of Atkin—Lehner quotients from the characteristic~$0$ consideration, which we discuss in Remark~\ref{rmk:comparison-RH-and-trace}. Reduction behavior of $X_0^*(N)$ modulo $p$ when $p\mid N$ is also studied in \cite{Padurariu--Park--Voight}, and we compare our result with theirs in Remark~\ref{rmk:comparison-PPV}.

Because the case is considerably simpler and in some sense more intuitive, for clarity and readability we include a discussion of the corresponding results specialized to the $X_0(pq)$ quotients. While we also describe $X_0(pq)$, $X_0(pq)/w_p$, and $X_0(pq)/w_q$ in text (see Propositions \ref{prop:X0pq}, \ref{prop:X0pqmodwp}, and \ref{prop:X0pqmodwq}),  we state the description of $X_0(pq)^*$ below.

\begin{thmx}[Theorem~\ref{thm:X0pqmodwpwq} in the main text]
Let $p,q$ be distinct primes $\geq 5$. The model of $X_0^*(pq)$ over $\overline{\bF}_p$ obtained from taking the quotient of the Deligne--Rapoport model of $X_0(pq)_{\overline{\bF}_p}$ by $\langle w_p,w_q\rangle$ is given by $X_0^+(q)_{\overline{\bF}_p}$ with
\begin{align*}
   &\frac12\left((g(X_0(pq)/w_p)-g(X_0(q)))-\frac12F(pq,1,p)+\frac12F(q,p,p)\right)\\
   =& \frac12\left((g(X_0(pq)/w_q)+1-2g(X_0^+(q)))-\frac12F(pq,1,p)-\frac12F(p,q,p)\right)
\end{align*}
self-crossings that are ordinary double points, where $F(q,p,p)$ is the number of fixed supersingular points of the $w_q$ action on $X_0(pq)_{\overline{\mathbb{F}}_p}$ and is given by \[F(q,p,p)=\frac{v(q)}{2}\left(1-\left(\frac{-q}{p}\right)\right),\]
and $F(pq,1,p)$ is the number of fixed supersingular points of the $w_{pq}$ action on $X_0(pq)_{\overline{\mathbb{F}}_p}$ and is given by \[ F(pq,1,p)=\frac{v(pq)}{2},\] if $pq\equiv 1 \pmod{4}$ or $p\equiv 1\pmod 4$ 
and $F(p,q,p)$ is the number of $w_p$-fixed supersingular points on $X_0(pq)_{\overline{\bF}_p}$ and is given by $$F(p,q,p)=\frac{v(p)}{2}\left(1+\left(\frac{-p}{q}\right)\right)$$ where $\left(\frac{-p}{q}\right)$ denotes the Legendre symbol and 
$$v(x)=\begin{cases}
    h(-4x) &\text{ for } x\equiv 1,2 \pmod{4},\\
    h(-x)+h(-4x)&\text{ for } x\equiv 3 \pmod{4}.
\end{cases}$$
Here and throughout,   $h(d)$ denotes the class number of the imaginary quadratic order of discriminant $d$.

Consequently, the genus of $X_0^*(pq)$ is equal to 
\begin{align*}
    &\frac12\left((g(X_0(pq)/w_p)-g(X_0(q)))-\frac12F(pq,1,p)+\frac12F(q,p,p)\right)+g(X_0^+(q))\\
    =&\frac12\left((g(X_0(pq)/w_q)+1-2g(X_0^+(q)))-\frac12F(pq,1,p)-\frac12F(p,q,p)\right)+g(X_0^+(q)).
\end{align*}
\end{thmx}

For the case where $N$ is simply one prime $p$, we also include this discussion and collect and include the corresponding equivalent statements from elsewhere in the literature cited appropriately.
\begin{thmx}[Theorem~\ref{thm:X0pmodwp} in the main text]
 The following are equal.
\begin{enumerate}
    \item[(a)] The genus of $X_0^+(p)/\mathbb{Q}$. 
    
    \item[(b)] The number of singular points on $X_0^+(p)_{\overline{\bF}_p}$ (model obtained by quotienting the Deligne--Rapoport model by $w_p$).
    
    \item[(c)] Half of the number of supersingular points on $X_0(p)$ over $\overline{\mathbb{F}}_p$ \textbf{not} fixed by $w_p$. 

    \item[(d)] Half of the number of supersingular $j$ invariants (all in $\mathbb{F}_{p^2}$) that are not in $\mathbb{F}_p$.
\end{enumerate}

When $p\ge 5$, it is also known that the total number $S_{p^2}$ of supersingular $j$-invariants (in $\mathbb{F}_{p^2}$) is given by 

\begin{equation*}
  S_{p^2}=  \left\lfloor\frac{p}{12}\right\rfloor+\begin{cases}
        0\quad\text{ if } p\equiv 1 \mod {12}\\
        1\quad\text{ if } p\equiv 5,7 \mod {12}\\
        2\quad\text{ if } p\equiv 11 \mod {12}
    \end{cases}
\end{equation*}
(\cite[Theorem V.4.1(c)]{Silverman}) and that the total number $S_p$ of supersingular $j$-invariants in $\mathbb{F}_p$ is given by
\begin{equation*}
   S_p= \begin{cases}
        \frac{1}{2}h(-4p)\quad\text{ if } p\equiv 1 \mod 4\\
        h(-p)\quad\text{ if } p\equiv 7 \mod 8\\
        2h(-p)\quad\text{ if } p\equiv 3 \mod 8,
    \end{cases}
\end{equation*}
 (\cite[Theorem 14.18]{Cox89}, as cited in \cite[p.2]{Delfs--Galbraith}). In particular, Item~(d) above is explicit when $p\ge 5$.
\end{thmx}

While the authors believe these genus formulae for Atkin--Lehner quotients of modular curves and the dual graph descriptions to be of independent interest, we note several instances in the literature where the dual graph may be an input for other applications. In \cite{Zhang93}, Zhang develops harmonic analysis on the metrized dual graph of the special fiber of a semistable model to construct an admissible height pairing. In \cite{Banerjee--Chaudhuri}, Banerjee and Chaudhuri apply Zhang's construction to semistable models of $X_0(p^2)$ to prove an effective version of Bogomolov conjecture for these modular curves.

Notably, there is a large program of provably computing all the rational points on all Atkin--Lehner quotients of $X_0(N)$. One of the techniques used to do this is Quadratic Chabauty, a depth 2 version of the Chabauty--Kim method \cite{Kim_2005,Kim_2009} with  implementations given in: \cite{BD1,BD2,BD3,BDMTV1,BDMTV2}. This technique requires computation of local $p$-adic height pairings at primes of bad reduction (see, for example, \cite[Theorem 3.2]{BDMTV2} which cites \cite{Betts--Dogra}, for a computation involving reduction graphs). As noted in \cite[Section 3.1, p.1116]{BDMTV2}, this computation also uses the harmonic analysis on metrized reduction graphs developed in \cite{Zhang93}. In some cases, such as \cite[Lemma 5.2]{BDMTV2}, it happens that the local heights at bad reduction places are trivial, but this is not the case for all Atkin--Lehner quotients of $X_0(N)$ for all square-free level $N$. In this way, our results can be viewed as providing information necessary for the computation of local heights.

This paper is organized as follows. In Section~\ref{sec:singularity}, we describe how singularities in a semistable model of a curve change when taking a finite quotient. In Section~\ref{sec:fixedpoints}, we provide counting formulae for $w_D$-fixed supersingular points on $X_0(N)_{\overline{\bF}_p}$. In Section~\ref{sec:X0+p}, we give a description of the model for $X_0^+(p)_{\overline{\bF}_p}$ obtained from quotienting the Deligne--Rapoport model of $X_0(p)_{\overline{\bF}_p}$ by $w_p$. In Section~\ref{sec:X0*pq}, we analyze quotients of $X_0(pq)_{\overline{\bF}_p}$ and obtain explicit descriptions for the model of $X_0^*(pq)_{\overline{\bF}_p}$ as well as the intermediate quotients. In Section~\ref{sec:square-free}, we describe any Atkin--Lehner quotient of $X_0(N)$ for square-free level $N$.

\section*{Acknowledgments}
The authors appreciate valuable conversations with Scott Carnahan, Nathan Chen, and John Voight during this project. We are also grateful to Abbey Bourdon, Wanlin Li, Jackson Morrow, Yunqing Tang, John Voight,  and David Zureick-Brown for helpful feedback on an earlier draft.

Nearing the completion of this project, the authors learned that Nikola Adžaga, Maarten Derickx, and Timo Keller were also working on a similar project at the same time, titled \emph{Semi-stable models, local heights and quadratic Chabauty for $X_0(N)^*$}, and we thank Abbey Bourdon for letting both groups know about the potential overlap. 

LB is grateful for the support of the U.S. National Science Foundation (DMS-2418835) and the Simons
Foundation (MPS-TSM–00007992). Part of this work was done while BW was visiting the Simons Laufer Mathematical Sciences Institute (SLMath) in Spring 2023, which was supported by the National Science Foundation (Grant No. DMS-1928930).

AI disclosure: The authors used Claude and ChatGPT for literature search and brainstorming. All text and proofs were written by the authors. 

\section{Singularities on quotients of curves}
\label{sec:singularity}

We construct models of Atkin--Lehner quotients of modular curves by considering how Atkin--Lehner operators act on models of the modular curve. In this process, we need to study how singularities behave under quotients.

\begin{lemma-definition}
\label{lemma-definition}
    Let $X$ be a nodal plane curve over the spectrum of a discrete valuation ring $\mathcal{O}_k$, and let $x$ be a closed point on the special fiber $X_s$ with nodal singularity. Let $G$ be a finite group acting on $X$ fixing $x$. Let $B=\mathcal{O}_{X,x}$ and let $\hat{B}$ be the completion of $B$. Let $t$ be a uniformizing parameter for $\mathcal{O}_k$ and for $u,v \in \hat{B}$, let $\tilde{u}, \tilde{v}$ be the images of $u,v$ mod $t$, respectively. Then $\hat{B}= \mathcal{O}_k[[u,v]]$ for some $u,v\in \hat{B}$ such that $uv-t^m=0$, $m\ge 1$ and there are precisely two minimal prime ideals of $\hat{B}/(t)$, which are $(\tilde{u})$ and $(\tilde{v})$, corresponding geometrically to the two local branches passing through $x$. The action of $G$ on $X$ fixing $x$ induces an action of $G$ on $\{(\tilde{u}),(\tilde{v})\}$. Let $G_0$ be the subgroup of $G$ containing the elements that fixes $(\tilde{u}),(\tilde{v})$ respectively. Then either $G_0=G$ or $G_0$ is an index-$2$ subgroup of $G$. If $G_0=G$, we call the action of $G$ on $X$ \emph{branch-preserving}. Otherwise, $G_0$ is an index-$2$ subgroup of $G$, in which case we call the action of $G$ on $X$ \emph{branch-swapping}.

    If $w$ fixes $(\tilde{u}),(\tilde{v})$ respectively, we call the action of $G$ on $X$ \emph{branch-preserving}. Otherwise, $w$ interchanges $(\tilde{u})$ and $(\tilde{v})$, in which case we call the action of $G$ on $X$ \emph{branch-swapping}.
\end{lemma-definition}

\begin{remark}
    The set-up of this Lemma-Definition is motivated by \cite[Chapter 10, Proposition 3.48]{Liu_ArithmeticCurve} (which cites \cite{Ray90}) and the proof of the claims in this Lemma-Definition are in the proof of the said proposition. Our assumption that $G$ fixes $x$ corresponds to the step of assuming $G$ is the inertia group $I$ in Qing Liu's proof.
\end{remark}

\begin{proposition}
\label{prop:quot_sing}
Let $S$ be the spectrum of a discrete valuation ring $\mathcal{O}_k$. Let $X$ be a semistable quasi-projective curve over $S$ and let $G$ be a finite group acting on $X$. Then the quotient $Y=X/G$ is semistable. In particular, let $x$ be a closed point of $X_s$ and let $y$ be its image in $Y_s$:
\begin{enumerate}
    \item If $X$ is smooth at $x$ then $Y$ is smooth at $y$.
    \item If $x$ is an ordinary double point on $X_s$  and $x$ has trivial stabilizer or $x$ has nontrivial stabilizer and the action of $G$ on $X$ fixing $x$ is branch-preserving then $y$ is an ordinary double point of $Y$.
    \item If $x$ is an ordinary double point on $X_s$ and $x$ has nontrivial stabilizer and the action of $G$ on $X$ fixing $x$ is branch-swapping then $y$ is a smooth point of $Y$.
\end{enumerate}
\end{proposition}
\begin{proof} For (1), see \cite[Chapter 10 Prop 3.48(a)]{Liu_ArithmeticCurve}.

For (2), in the case where $x$ has trivial stabilizer, let $\pi\colon X \to Y$ denote the quotient morphism, and set $y = \pi(x)$. 
Since $x$ has trivial stabilizer, the decomposition group $G_d(x)$ is trivial. 
Taking $H = \{1\} = $ in \cite[Expos\'e V, Proposition 2.2]{Grothendieck} gives an isomorphism of completed local rings at $x$ and at $y$, and reducing modulo a uniformizer of $\mathcal{O}_K$ gives an isomorphism on the special fiber, namely we have that $\widehat{\mathcal{O}}_{Y_s,y} \cong \widehat{\mathcal{O}}_{X_s,x}$.
Since $x$ is an ordinary double point of $X_s$, this isomorphism asserts that $y$ is an ordinary double point of $Y_s$.

In the case where $x$ has nontrivial stabilizer and in case (3), see the proof of Proposition 3.48 (b) in \cite[Page 527]{Liu_ArithmeticCurve}.
\end{proof}

\begin{remark}
\label{rmk:may-not-regular}
    Note from the proof on \cite[p.527] {Liu_ArithmeticCurve} that the thickness (minimal $m$ in Lemma-Definition~\ref{lemma-definition})  of the nodal singularities may increase after quotienting by $G$. As a result, the semistable models we obtain from quotienting the Deligne--Rapoport model by Atkin--Lehner involutions may not be regular even at the image of regular nodal singularities in the Deligne--Rapoport model. 
\end{remark}

In the following sections, we often refer to an ordinary double point as a \emph{crossing} to shorten the sentences. 

Recall also that the dual graph of a semistable curve (reduced curve whose only singularities are ordinary double points) is obtained by constructing a vertex for each irreducible component of the curve, and an edge connecting vertex $i$ with vertex $j$ for each crossing of the corresponding irreducible components $C_i$ and $C_j$ (including the case where $i=j$). We shall use the following proposition in the computations of genera in the later sections.

\begin{proposition}
\label{prop:genus-dual-graph}
\cite[Chapter 10 Lemma 3.18.]{Liu_ArithmeticCurve}The arithmetic genus $p_a(X_{\eta})$ of a semistable curve $X$ with irreducible components $C_1,\dotsc, C_n$ of $X_s$ and dual graph $G$ is given by $$p_a(X_{\eta})=\beta(G)+\sum_{i=1}^n p_a(C_i'),$$ where $\beta(G)$ denotes the Betti number of the graph $G$ and $C_i'$ is the normalization of $C_i$.

In particular, if $X_s$ is obtained from a smooth curve $C$ crossing itself at $n$ ordinary double points, then $p_a(X_{\eta})=n+g(C)$, where $g(C)$ is the geometric genus of $C$. If $X_s$ is obtained from two smooth curves $C_1, C_2$ crossing at $n$ ordinary double points, then $p_a(X_{\eta})=n-1+g(C_1)+g(C_2)$, where $g(C_1),g(C_2)$ are the geometric genus of $C_1,C_2$ respectively. [See Figure~\ref{fig:special} for an illustration of these two special cases.]
\end{proposition}

\begin{figure}[h]
\includegraphics[scale=0.2]{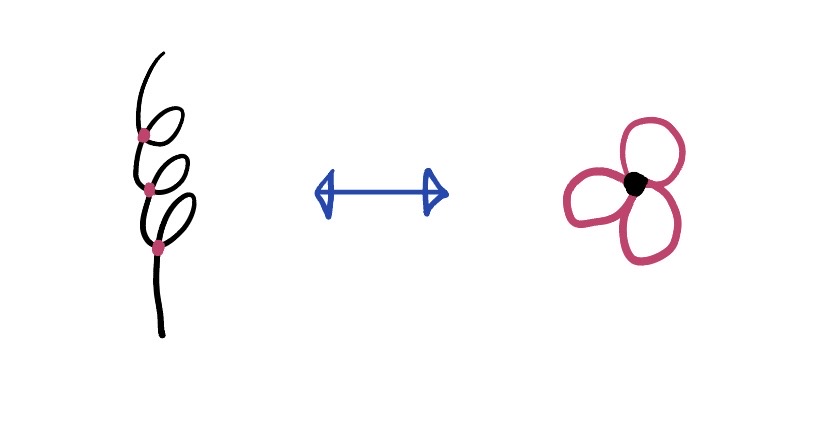}
\includegraphics[scale=0.5]{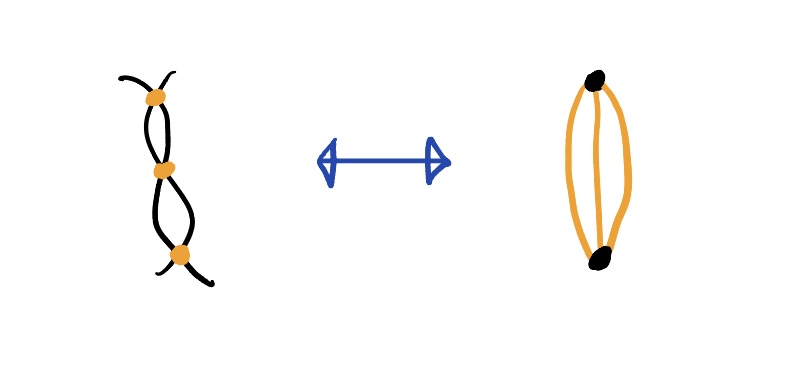}
\caption{two special cases in Proposition 2.4}
\label{fig:special}
\end{figure}

\begin{remark}
    \label{rmk:arith-geom-genus} Given a smooth, projective, and geometrically connected curve $X$ over a field, we have that its geometric genus is equal to its arithmetic genus \cite[Chapter 10 Remark 1.7]{Liu_ArithmeticCurve}, and thus we can talk about the genus $g(X)$ without ambiguity. Moreover, the arithmetic genus is preserved under reduction: given a projective flat integral model of $X$, $g(X)$ is equal to the arithmetic genus of any fiber \cite[Chapter 8, Corollary 3.6(a)]{Liu_ArithmeticCurve}.
\end{remark}

\section{Supersingular points fixed by Atkin--Lehner involutions}
\label{sec:fixedpoints}

When examining the special fibers at a prime $p$ of Atkin--Lehner quotients of modular curves, we will need to count the number of supersingular points fixed by an Atkin--Lehner involution $w_D$. In this section, we develop this count, summarized in Proposition~\ref{prop:fixed-point-formula} at the end of this section. Since the supersingular locus of $X_0(N)$ is disjoint from the cusps, here and for the rest of the paper, we do not need to discuss generalized elliptic curves.

\begin{lemma}
\label{lem:norm_D_elt}
    Let $D>1$ be a square-free integer. An order $\cO_f$ of conductor $f$ in $\mathbb{Q}(\sqrt{-D})$ contains an element of norm $D$ if and only if we have one of the following two cases:
    \begin{itemize}
        \item $D\equiv 3\pmod4$ and ($f=1$ or $f=2$); or
        \item $D\equiv 1,2\pmod4$ and $f=1$.
    \end{itemize}
    This happens if and only if $\sqrt{-D}\in \cO_f$.
\end{lemma}

\begin{proof}
     Recall that an order $\mathcal{O}_f$ of conductor $f$ in $K=\mathbb{Q}(\sqrt{-D})$ has an element of norm $D$ if and only if $N(\alpha)=D$ has a solution in $\mathcal{O}_f$, where $N(\alpha)$ denotes the norm of $\alpha$. Let $d_K$ denote the discriminant of $K=\mathbb{Q}(\sqrt{-D})$. Then any $\alpha \in O_K$ can be written as $\alpha=x+y\omega$ for some $x,y\in\mathbb{Z}$, where $\omega=\frac{d_K+\sqrt{d_K}}{2}$. (For this and other facts about quadratic imaginary orders used here, see \cite[Section 7]{Cox89}.) Note that $$N(\alpha)=N(x+y\omega)=\left(x+\frac{d_Ky}{2}\right)^2-\frac{d_Ky^2}{4}=x^2+d_Kxy+\frac{(d_K^2-d_K)y^2}{4}.$$ 
Moreover, $\alpha \in \mathcal{O}_f$ if and only if $f\mid y$. Therefore, $\mathcal{O}_f$ contains an element of norm $D$ if and only if there exist $x, y\in\mathbb{Z}$ such that $f\mid y$ and $N(x+y\omega)=D$.

If $D\equiv 3\pmod{4}$ then $d_K=-D$, and $N(x+y\omega)=x^2-Dxy+\frac{(D^2+D)y^2}{4}$.  We look for $x,y\in\mathbb{Z}$ such that $x^2-Dxy+\frac{(D^2+D)y^2}{4}=D$. The discriminant of this quadratic equation as a function of $x$ is $D^2y^2-D^2y^2-Dy^2+4D=-Dy^2+4D$, which is negative when $y^2>4$. It remains to discuss the existence of such $x\in \mathbb{Z}$ when $y=0,\pm 1, \pm 2$. If $y=0$, there is no $x\in\mathbb{Z}$ satisfying the equation since $D$ is square-free. If $y=\pm 1$, the equation is equivalent to $4x^2\mp 4Dx+D^2-3D=0$, i.e. $(2x\mp D)^2=3D$, which has an integer solution for $x$ if and only if $D=3$. If $y=\pm 2$, the equation is equivalent to $x^2\mp 2Dx+D^2=0$, which is equivalent to $(x\mp D)^2=0$ and always has an integer solution for $x$. 

If $D\equiv 1,2 \pmod{4}$ then $d_K=-4D$, and $N(x+y\omega)=x^2-4Dxy+(4D^2+D)y^2$. We look for $x,y\in\mathbb{Z}$ such that $x^2-4Dxy+(4D^2+D)y^2=D$. The discriminant of this quadratic equation as a function of $x$ is $16D^2y^2-16D^2y^2-4Dy^2+4D=-4Dy^2+4D$, which is negative when $y^2>1$. It remains to discuss the existence of such $x\in\mathbb{Z}$ when $y=0,\pm 1$. If $y=0$, there is no $x\in\mathbb{Z}$ satisfying the equation since $D$ is square-free. If $y= \pm1$, the equation is equivalent to $x^2\mp 4Dx+4D^2=0$, which is equivalent to $(x\mp 2D)^2=0$ and always has an integer solution for $x$.

Therefore, there exist $x, y\in\mathbb{Z}$ such that $f\mid y$ and $N(x+y\omega)=D$ if and only if ($D\equiv 3\pmod4$ and ($f=1$ or $f=2$)) or ($D\equiv 1,2\pmod4$ and $f=1$), hence the first part of the theorem.

When $D\equiv 3\pmod{4}$, $K$ has ring of integers $\cO_K=\mathbb{Z}\left[\frac{1+\sqrt{-D}}{2}\right]$, since $\cO_f$ is of index $f=1$ or $2$ in $\cO_K$, it contains the element $\sqrt{-D}$. When $D\equiv 1,2\pmod{4}$, $K$ has ring of integers $\cO_K=\mathbb{Z}\left[\sqrt{-D}\right]$, since $\cO_f$ is of index $f=1$ in $\cO_K$ and thus must be $\cO_K$ itself, it contains the element $\sqrt{-D}$ as well.  Conversely, if $\sqrt{-D}\in \cO_f$, then since $\sqrt{-D}$ has norm $D$, $\cO_f$ contains an element of norm $D$, completing the proof of the last claim.
\end{proof}

Given a supersingular elliptic curve $E$ over $\overline{\mathbb{F}}_{p}$, the unique cyclic subgroup scheme $C_p$ of order $p$ in $E$ is given by the kernel of the $p$-power Frobenius $E\to E^{(p)}$. (\cite[V.3.1]{Silverman} and \cite[p.94]{Katz--Mazur}) By the Chinese remainder theorem, any cyclic order-$pM$ subgroup scheme $C_{pM}$ of $E$ can be written as $C_{pM}=C_p+C_M$ for some uniquely determined subgroup scheme $C_M$ of order $M$ in $E$. This sets us up for the following lemma.

\begin{lemma}
\label{lem:M_vs_pM}
Let $p$ be a prime number and let $D, m$ be square-free positive integers such that $(D,m)=1$, $p\nmid D$ and $p\nmid m$. Let $M=Dm$. Then the map $[(E,C_M)]\mapsto [(E,C_M+C_p)]$ defines a bijection from the set of supersingular points on $X_0(M)_{\overline{\mathbb{F}}_p}$ to the set of supersingular points on $X_0(pM)_{\overline{\mathbb{F}}_p}$. Moreover, $w_D$ fixes $[(E,C_M)]$ if and only if it fixes $[(E,C_M+C_p)]$.
\end{lemma}

\begin{proof}
     It remains to show that $w_D$ fixes $[(E,C_M)]$ if and only if it fixes $[(E,C_{p}+C_M)]$.

     Recall that $w_D$ sends $[(E, C_M)]$ to $[(E/C_D, (C_M+E[D])/C_D)]$ and  $[(E, C_{p}+C_M)]$ to $[(E/C_D, (C_p+C_M+E[D])/C_D)]$, where $C_D$ is the unique order $D$ subgroup of $C_M$. Meanwhile, given any isomorphism $\phi: E\to E/C_D$, $\phi$ maps $C_p$ to the unique order $p$ subgroup scheme $(C_p+C_D)/C_D$ of $E/C_D$. Therefore, if $\phi$ maps $C_M$ to $(C_M+E[D])/C_D$, then it maps $C_p+C_M$ to $(C_p+C_D+C_M+E[D])/C_D=(C_{p}+C_M+E[D])/C_D$. Conversely, if $\phi$ maps $C_p+C_M$ to $(C_p+C_M+E[D])/C_D$, then it maps the unique order-$M$ subgroup $C_M$ of $C_p+C_M$ to the unique order-$M$ subgroup of $(C_p+C_M+E[D])/C_D$, which is $(C_M+E[D])/C_D$. This shows that $w_D$ fixes $[(E,C_M)]$ if and only if it fixes $[(E,C_p+C_M)]$, as desired.
\end{proof}

\begin{proposition}
\label{prop:2-1-p-nmid-D}
     Let $p\ge 5$ be a prime and let  $D\ge 5, m\ge 1$ be square-free integers such that $(D,m)=1$, $p\nmid D$ and $p\nmid m$. Let $M=Dm$. Let $E$ be a supersingular elliptic curve over $\overline{\mathbb{F}}_p$ and let $R:=\End(E)$. The function

$$\Phi:\left\{(\alpha, C_m): \begin{array}{c}
   \alpha\in R, \trd(\alpha)=0,\nrd(\alpha)=D,\\ C_m \text{ an order-}m \text{ cyclic subgroup }\\ \text{scheme of } E,
    \alpha(C_m)=C_m 
\end{array}\right\}/R^{\times}\to \left\{\begin{array}{c}
     w_D\text{-fixed points associated with } E \\
     \text{on }X_0(M)_{\overline{\mathbb{F}}_p}
\end{array}  \right\}$$ 
defined by $\Phi[(\alpha,C_m)]=[(E,\ker(\alpha) + C_m)]$ is a well-defined, surjective, $2$-to-$1$ map. Here, the $R^{\times}$ action is given by $u\cdot (\alpha,C_m)=(u\alpha u^{-1}, uC_m)$. Moreover, the fiber above $[(E,\ker(\alpha)+C_m)]$ is precisely $\{[\alpha,C_m],[-\alpha,C_m]\}$ where $[\alpha,C_m]\ne [-\alpha,C_m]$.
\end{proposition}

\begin{remark}
        By \cite{Deuring1941}, as cited in \cite[Theorem 42.1.9]{Voight}, $\End(E)$ of a supersingular elliptic curve $E$ over $\overline{\mathbb{F}}_p$ is a maximal order in the quaternion algebra $B_{p,\infty}$, ramifying precisely at $p$ and $\infty$. The norm and trace of an element $\alpha$ in $R$ refers to the reduced norm $\nrd(\alpha)=\alpha\bar{\alpha}=\bar{\alpha}\alpha$ and the reduced trace $\trd(\alpha)=\alpha+\bar{\alpha}$ in $B_{p,\infty}$ and the element $\bar{\alpha}$ corresponds to the dual isogeny of $\alpha$. 
\end{remark}

\begin{proof}[Proof of Proposition~\ref{prop:2-1-p-nmid-D}]
    First, we will show well-definedness of the map $\Phi$. Let $(\alpha,C_m)$ be in the domain of $\Phi$. Then $\alpha$ defines a degree-$D$ isogeny from $E$ to $E$ and thus $\ker(\alpha)$ is an order-$D$ subgroup scheme of $E$. Since $(m,D)=1$, the sum $\ker(\alpha)+C_m$ is a direct sum, giving an order-$M$ subgroup scheme of $E$ (automatically cyclic since $M$ is square-free). Recall that $w_D[(E,C_M)]=[(E/C_D,(C_M+E[D])/C_D)]$, where $C_D$ is the unique order-$D$ subgroup of $C_M$. So $w_D[(E,\ker(\alpha)+C_m))]=[(E/\ker(\alpha),C_m+E[D]/\ker(\alpha))]$. To see that $[(E,\ker(\alpha)+C_m))]$ is a fixed point of $w_D$, it suffices to show that the isomorphism $\tilde{\alpha}: E/\ker(\alpha)\to E$ induced by $\alpha$ sends $(C_m+E[D])/\ker(\alpha)$ to $\ker(\alpha)+C_m$.  Indeed, $\tilde{\alpha}((C_m+E[D])/\ker(\alpha))=\alpha(C_m+E[D])=\alpha(C_m)+\alpha(E[D])=C_m+\ker(\alpha)$. This proves that the outputs of $\Phi$ are $w_D$-fixed points. Moreover, if $[(\alpha,C_m)]=[(\alpha',C_m')]$, then $\alpha'=u\alpha u^{-1}$ for some $u\in R^{\times}=\Aut(E)$ and $C_m'=uC_m$. So $\ker(\alpha')=u\ker(\alpha)$, and therefore the automorphism $u$ identifies $(E,\ker(\alpha)+C_m)$ and $(E,\ker(\alpha')+C_m')$ in the same isomorphism class, completing the proof of well-definedness.
    
    For surjectivity of $\Phi$, let $[(E,C_M)]$ be a fixed point of $w_D$ and let $C_D$ (resp. $C_m$) be the unique order-$D$ (resp. order-$m$) subgroup scheme of $C_M$. Then there exists an isomorphism $\phi: E\to E/C_D$ such that $\phi$ sends $C_M$ to $(C_M+E[D])/C_D$. Consequently, $\phi$ sends $C_D$ (resp. $C_m$) to $E[D]/C_D$ (resp. $(C_m+C_D)/C_D$), the unique order-$D$ (resp. order-$m$) subgroup schemes of $(C_M+E[D])/C_D$. Consider the natural projection $\pi: E\to E/C_D$ and let $\alpha:=\phi^{-1}\circ\pi$. Then $\alpha$ is an endomorphism on $E$ such that $\ker{\alpha}=C_D$. So $\alpha$ is a degree-$D$ isogeny on $E$. Moreover, $\alpha^2(E[D])=\alpha\left(\phi^{-1}\left(\pi(E[D])\right)\right)=\alpha\left(\phi^{-1}\left(E[D]/C_D\right)\right)=\alpha\left(C_D\right)=0$, so $E[D]$ is an order $D^2$ subgroup scheme of $\ker(\alpha^2)$ and thus must equal $\ker(\alpha^2)$. Therefore, $\alpha^2=u\circ [D]$ for some automorphism $u$ on $E$, where $[D]$ denotes the multiplication by $D$ map. Let $t:=\alpha+\overline{\alpha}$ be the trace of $\alpha$. Then $t\alpha=\alpha^2+\overline{\alpha}\circ\alpha=(u+1)\circ[D]$. Since $\deg(t\alpha)=t^2D$ and $\deg((u+1)\circ[D])=\deg(u+1)D^2$, we have $D\mid t^2$. Since $D$ is square-free, we must have $D\mid t$, i.e. $t=nD$ for some integer $n$. Since for any $x,y\in\mathbb{Z}$,  $x^2-xyt+y^2D=\deg(x[1]-y\alpha)\ge0$, we have $t^2-4D\le 0$ and thus $n^2D\le 4$. Since $D\ge 5$, we obtain $n=0$ and thus $t=0$, $u=-1$. Thus $\alpha$ has trace $0$ and norm $D$ as desired. Meanwhile, $\alpha(C_m)=\phi^{-1}\pi(C_m)=\phi^{-1}((C_m+C_D)/C_D)=C_m$, so $[(\alpha,C_m)]$ is in the domain of $\Phi$ and $\Phi[(\alpha,C_m)]=[(E,C_D+C_m)]=[(E,C_M)]$.

    Finally, we show that the function is $2$-to-$1$ with the claimed fiber. Suppose $[(E,C_M)]=\Phi[(\alpha,C_m)]=\Phi[(\beta,C_m')]$. Then there exists an isomorphism $u:E\to E$ such that $u(\ker(\alpha)+C_m)=\ker(\beta)+C_m'$. Since both $u(\ker(\alpha))$ and $\ker(\beta)$ are subgroup schemes of order $D$ in $\ker(\beta)+C_m'$, we must have  $u(\ker(\alpha))=\ker(\beta)$ and similarly, $u(C_m)=C_m'$. Choosing a different representative of $[(\alpha,C_m)]$ if necessary, we may assume without loss of generality that $\ker(\alpha)=\ker(\beta)$ and $C_m=C_m'$. Then $\beta=\gamma\circ \alpha$ for some $\gamma\in \Aut(E)$. By \cite[III.10.2]{Silverman}, when $p\ge 5$, 
    \begin{equation*}
    \Aut(E)=\begin{cases}
            \{\pm 1\}\quad\text{ if }j(E)\ne 1728, 0\\
            \text{cyclic group of order 4 \quad if }j(E)=1728\\
            \text{cyclic group of order 6 \quad if }j(E)=0\\
        \end{cases}
    \end{equation*}
    If $j(E)\ne 1728, 0$, then $\Aut(E)=\{\pm 1\}$ and conjugation by $\Aut(E)$ on $\End(E)$ is trivial, so $\beta=\pm\alpha$ and $[(\alpha,C_m)]\ne [(-\alpha,C_m)]$, proving the function $\Phi$ is $2$-to-$1$ with the claimed fiber in this case.

    Before we discuss the cases $j(E)=1728$ or $0$, first notice that any element $\gamma\in \Aut(E)=R^{\times}$ have norm $1$, so $\overline{\gamma}=\gamma^{-1}$. Meanwhile, $\trd(\gamma\alpha)=\gamma\alpha+\overline{\gamma\alpha}=\gamma\alpha+\bar{\alpha}\bar{\gamma}=\gamma\alpha-\alpha\overline{\gamma}=\gamma\alpha-\alpha \gamma^{-1}$. We also have the following equivalences: $\trd(\gamma\alpha)=0\iff \gamma\alpha=\alpha\gamma^{-1}\iff \alpha\gamma=\gamma^{-1}\alpha\iff \trd(\gamma^{-1}\alpha)=0$, which we shall use in the cases below.

    If $j(E)=1728$, let $g$ be a generator of $\Aut(E)$. Then $g$ is of order $4$ and norm $1$, so $g^2=-1$ and $\overline{g}=g^{-1}=-g$. We discuss whether $\trd(g\alpha)=0$.
    \begin{itemize}
        \item  If $\trd(g\alpha)=0$, then $g\alpha=\alpha g^{-1}=-\alpha g$. Thus by \cite[Lemma 2.2.5]{Voight}, $\{1,g,\alpha,g\alpha\}$ is a $\bQ$-basis of  $R\otimes\bQ$, which is isomorphic to $\left(\frac{-1,-D}{\bQ}\right)$, splitting at $p$ since $p$ is odd and $p\nmid D$ \cite[12.4.12 (a)]{Voight}. This contradicts $R\otimes\bQ\cong B_{p,\infty}$, so the case is impossible.

        \item If $\trd(g\alpha)\ne 0$, then $\trd(g^3\alpha)=\trd(g^{-1}\alpha)\ne 0$, so $\beta$, as a trace-$0$ element, can only be $\alpha$ or $g^2\alpha=-\alpha$. Meanwhile, $-\alpha$ and $\alpha$ are non-conjugate. Indeed, $(\pm 1)\alpha(\pm 1)^{-1}=\alpha\ne -\alpha$, and $(g^3)\alpha(g^3)^{-1}=(-g)\alpha(-g)^{-1}=g\alpha g^{-1}\ne -\alpha$ because otherwise $\trd(g\alpha)=g\alpha-\alpha g^{-1}=g\alpha+g^2\alpha g^{-1}=g(\alpha+g\alpha g^{-1})=0$, a contradiction. So $[(\alpha,C_m)]$ and $[(-\alpha,C_m)]$ are the two and only two elements in the fiber above $[(E,\ker(\alpha)+C_m)]$.
        
    \end{itemize}

   If $j(E)=0$, let $\sigma$ be a generator of $\Aut(E)$. Then $\sigma$ is of order $6$ and norm $1$, so $\sigma^3=-1$ (and thus $\sigma^2=-\sigma^{-1}$) and $\overline{\sigma}=\sigma^{-1}$. From our earlier analysis, $\trd(\sigma\alpha)=0\iff \trd(\sigma^{-1}\alpha)=0\iff \trd(-\sigma\alpha)=0\iff \trd(-\sigma^{-1}\alpha)=0$. So $\sigma\alpha,\sigma^2\alpha=-\sigma^{-1}\alpha, \sigma^{4}\alpha=-\sigma{\alpha}, \sigma^{5}\alpha=\sigma^{-1}\alpha$ are either all trace $0$, or all trace nonzero. Moreover, since $0=\sigma^3+1=(\sigma+1)(\sigma^2-\sigma+1)$ and $\sigma\ne -1$, we have $\sigma^2=\sigma-1$.

    \begin{itemize}
        \item  If $\trd(\sigma\alpha)=0$, then $\sigma\alpha=\alpha\sigma^{-1}$ and $\sigma^{-1}\alpha=\alpha\sigma$. Consider $s=\sigma-\sigma^{-1}$, then $s\alpha=\sigma\alpha-\sigma^{-1}\alpha=\alpha\sigma^{-1}-\alpha\sigma=-\alpha s$. Meanwhile, $s^2=\sigma^2-2+\sigma^{-2}=\sigma^2-2-\sigma=-3$. Therefore, similar to the $j(E)=1728$ case, $\{1,s,\alpha,s\alpha\}$ is a $\bQ$-basis of  $R\otimes\bQ$, which is isomorphic to $\left(\frac{-3,-D}{\bQ}\right)$, splitting at $p$ since $p\ge 5$ and $p\nmid 3D$. This contradicts $R\otimes\bQ\cong B_{p,\infty}$, so the case is impossible.
       
        \item If $\trd(\sigma\alpha)\ne 0$, then the only  $\beta=\sigma^{i}\alpha$ with trace $0$ are $\alpha$ and $\sigma^3\alpha=-\alpha$.  Suppose $-\alpha=\gamma\alpha\gamma^{-1}$ for some $\gamma\in \Aut(E)$, then $-\alpha\gamma=\gamma\alpha$, thus $\trd(\gamma\alpha)=-\alpha\gamma-\alpha(\bar{\gamma})=-\alpha\trd(\gamma)=0$ since $\bQ\cap \alpha\bQ=\{0\}$. So $\trd(\gamma)=0$, which is impossible for $\gamma\in \Aut(E)\cong \mu_6$. This shows that $[(\alpha,C_m)]$ and $[(-\alpha,C_m)]$ are the two and only two elements in the fiber above $[(E,\ker(\alpha)+C_m)]$. 
         
    \end{itemize}
\end{proof}

\begin{proposition}
\label{prop:2-1-p-mid-D}
Let $p\ge 5$ be a prime and let  $D\ge 5, m\ge 1$ be square-free integers such that $(D,m)=1$ and $p\mid D$ (and thus $p\nmid m$). Let $M=Dm$.
Let $E$ be a supersingular elliptic curve over $\overline{\mathbb{F}}_p$ and let $R:=\End(E)$. Under these new assumptions, the function $\Phi$ defined in Proposition~\ref{prop:2-1-p-nmid-D} is still a well-defined, surjective, $2$-to-$1$ map. Moreover, the fiber above $[(E,\ker(\alpha)+C_m)]$ is precisely $\{[\alpha,C_m],[-\alpha,C_m]\}$ where $[\alpha,C_m]\ne [-\alpha,C_m]$, unless $j(E)=1728$ and $\trd(g\alpha)=0$, where $g$ denotes a generator of $\Aut(E)$. In the case where $j(E)=1728$ and $\trd(g\alpha)=0$, the two distinct elements in the fiber above $[(E,\ker(\alpha)+C_m)]$ are $[(\alpha,C_m)]$ and $[(g\alpha,C_m)]$.
\end{proposition}

\begin{proof}
Under this set-up, the well-definedness and surjectivity of $\Phi$ follows exactly as in the proof of Proposition~\ref{prop:2-1-p-nmid-D}. The analysis of fibers and consequently proving the map is $2$-to-$1$ in Proposition~\ref{prop:2-1-p-nmid-D} works except for the $\trd(g\alpha)=0$ case when $j(E)=1728$, and the $\trd(\sigma\alpha)=0$ case when $j(E)=0$, which fail because $p\mid D$. We therefore treat these two cases separately here.

    If $j(E)=0$ and $\trd(\sigma\alpha)=0$, then for each integer $i$ we have $\trd(\sigma^i\alpha)=0$, so the distinct choices for $\beta$ are $\sigma^{i}\alpha$ for $i=0,1,2,\dotsc, 5$. Since $\trd(\sigma\alpha)=\sigma\alpha-\alpha\sigma^{-1}=0$, we have $\sigma\alpha=\alpha\sigma^{-1}$. Thus $\sigma\alpha\sigma^{-1}=\sigma^2\alpha$ and $\sigma^{2}\alpha\sigma^{-2}=\sigma(\sigma^2\alpha)\sigma^{-1}=\sigma^{4}\alpha$. Since $\pm 1, \pm \sigma, \pm \sigma^2$ are the only six elements of $\Aut(E)$ in this case, and conjugation by $\pm \gamma$ acts the same, we obtain the conjugacy class of $\alpha$ under $\Aut(E)$ is precisely $\{\alpha, \sigma^2\alpha, \sigma^4\alpha\}$. Similarly, $\{\sigma\alpha, \sigma^3\alpha, \sigma^5\alpha\}$ form the other conjugacy class. Note that $\sigma^3\alpha=-\alpha$. So we again arrive at $[(\alpha,C_m)]\ne [(-\alpha,C_m)]$ are the two and only two elements in the fiber above $[(E,\ker(\alpha)+C_m)]$. 

If $j(E)=1728$ and $\trd(g\alpha)=0$, then $\trd(g^3\alpha)=0$ as well. So $\beta=\alpha, g\alpha, -\alpha,$ or $-g\alpha$. Since $\trd(g\alpha)=0$, we have $g\alpha=\alpha g^{-1}$, and thus $-\alpha=g^2\alpha=g\alpha g^{-1}$. Therefore, $[-\alpha]=[\alpha]$. Similarly, $[-g\alpha]=[g\alpha]$. It remains to show $[\alpha]\ne[g\alpha]$. Indeed, $(\pm 1)\alpha(\pm 1)^{-1}=\alpha\ne g\alpha$, and $(g^3)\alpha(g^3)^{-1}=(-g)\alpha(-g)^{-1}=g\alpha g^{-1}\ne g\alpha$ because otherwise $g^{-1}=1$, contradicting $g$ has order $4$. So $[(\alpha,C_m)]\ne [(g\alpha,C_m)]$ are the two and only two elements in the fiber above $[(E,\ker(\alpha)+C_m)]$.
\end{proof}

The following lemma is clear.
\begin{lemma}
\label{lem:bij-cm}
    Let $p$ be a prime and $D$ be a positive integer. Let $E$ be a supersingular elliptic curve over $\overline{\mathbb{F}}_p$. Let $R:=\End(E)$, a maximal order in $B_{p,\infty}=\End^0(E)$. The map from $\{\alpha\in R \mid \trd(\alpha)=0,\nrd(\alpha)=D\}$ to $\{\iota: \mathbb{Q}(\sqrt{-D})\hookrightarrow B_{p,\infty}\mid \iota(\sqrt{-D})\in R\}$ sending $\alpha$ to $\iota_{\alpha}:a+b\sqrt{-D}\mapsto a+b\alpha$ is a bijection, which induces a bijection on $R^{\times}$ conjugacy classes. 
\end{lemma}

\begin{remark}
    Let $p$ be a prime and $D$ be a positive integer. By the Albert--Brauer--Hasse--Noether theorem as in \cite[Proposition 14.6.7]{Voight}, $K=\mathbb{Q}(\sqrt{-D})$ embeds into $B_{p,\infty}$ if and only if $p$ does not split in $K$, i.e. $\left(\frac{d_K}{p}\right)\ne 1$. If $p$ is further assumed to be odd and $p\nmid D$, this is equivalent to $\left(\frac{-D}{p}\right)=-1$. In particular, the sets involved in the above lemmas and propositions can be empty.
\end{remark}

Following the set-up of Lemma~\ref{lem:bij-cm}, let $\cO_{\alpha}:=\iota_{\alpha}^{-1}(\mathbb{Q}(\alpha)\cap R)$. Then $\cO_{\alpha}$ is an order in $\mathbb{Q}(\sqrt{-D})$ containing $\sqrt{-D}$, and $\iota_{\alpha}$ restricted to $\cO_{\alpha}$ is an optimal embedding (see \cite[Definition 30.3.2]{Voight}) from $\cO_{\alpha}$ to $R$. 

\begin{proposition}
\label{prop: drop-m-cyclic}
    Let $p$ be a prime and let  $D\ge 5, m\ge 1$ be square-free integers such that $(D,m)=1$ and $p\nmid m$. Let $E$ be a supersingular elliptic curve over $\overline{\mathbb{F}}_p$ and let $R:=\End(E)$. The function 

$$\Psi:\left\{(\alpha, C_m): \begin{array}{c}
   \alpha\in R, \trd(\alpha)=0,\nrd(\alpha)=D,\\ C_m \text{ an order-}m \text{ cyclic subgroup }\\ \text{scheme of } E,
    \alpha(C_m)=C_m 
\end{array}\right\}/R^{\times}\to \left\{\alpha \in R\ : \trd(\alpha)=0,\nrd(\alpha)=D\right\}/R^{\times}$$ 
defined by $\Psi[(\alpha,C_m)]=[\alpha]$ is a well-defined map. Here, the $R^{\times}$ actions are given by $u\cdot (\alpha,C_m)=(u\alpha u^{-1}, uC_m)$ and $u\cdot\alpha=u\alpha u^{-1}$ respectively for the domain and codomain. Moreover, fixing an equivalence class $[(\alpha,C_m)]$ in the domain, the set map $\psi:[(\alpha,C_m)]\to [\alpha], (\beta,C)\mapsto \beta$ is a bijection.

Let $\nu(m,\alpha)$ denote the number of order-$m$ cyclic subgroup schemes $C_m$ of $E$ stabilized by $\alpha$. (Here and throughout, we say $C_m$ is stabilized by $\alpha$, or $\alpha$-stable, to mean $\alpha(C_m)=C_m$.) Then for any $[\alpha]$ in the codomain, the preimage set $\Psi^{-1}([\alpha])$ contains precisely $\nu(m,\alpha)$ elements. Moreover, $$\nu(m,\alpha)=\prod_{\substack{\ell \mid m\\ \ell \text{ is prime }}}\nu(\ell,\alpha),$$ where $$\nu(\ell,\alpha)=\left(1+\left(\frac{-D}{\ell}\right)\right)$$ if $\ell$ is odd; if $\ell=2$ is a factor of $m$,  $$\nu(2,\alpha)=\begin{cases}
    3\quad\text{ if } \alpha \text{ acts as identity on }E[2],\\
    1\quad\text{ otherwise.}
\end{cases}$$ 

\end{proposition}

\begin{proof}
 First, the well-definedness of $\Psi$ is clear from the definition and the $R^{\times}$ actions. 
 
 Before we analyze $\#\Psi^{-1}([\alpha])$, we first prove the set map $\psi:[(\alpha,C_m)]\to [\alpha], (\beta,C)\mapsto \beta$ is a bijection. Indeed, $\psi$ is clearly well-defined and surjective. Assume to the contrary that $\psi$ is not injective, then there exists some $(\beta,C),(\beta,C')\in [(\alpha,C_m)]$ such that $C\ne C'$. Thus $(\beta,C')=u\cdot(\beta,C)=(u\beta u^{-1}, uC)$ for some $u\in R^{\times}$. In particular, $u$ is in the centralizer of $\beta$. Since the centralizer of $\beta$ in the quaternion algebra $B_{p,\infty}$ is $\bQ(\beta)$, we obtain $u\in \bQ(\beta)\cap R^{\times}$. Thus $\iota_{\beta}^{-1}(u)\in \cO_{\beta}^{\times}\subseteq \cO_K^{\times}$. Since $K=\bQ(\sqrt{-D})$ for $D\ge 5$, $\cO_K^{\times}=\{\pm 1\}$. Thus $u\in\{\pm 1\}$ and $C'=uC=C$, a contradiction.

 Since each $\psi$ is a bijection, the number of elements in the fiber $\Psi^{-1}([\alpha])$ is the same as the number of elements in the fiber $\widetilde{\Psi}^{-1}(\alpha)$, where $\widetilde{\Psi}$ denotes the map $$\left\{(\alpha, C_m): \begin{array}{c}
   \alpha\in R, \trd(\alpha)=0,\nrd(\alpha)=D,\\ C_m \text{ an order-}m \text{ cyclic subgroup }\\ \text{scheme of } E,
    \alpha(C_m)=C_m 
\end{array}\right\}\to \left\{\alpha \in R\ : \trd(\alpha)=0,\nrd(\alpha)=D\right\},(\alpha,C_m)\mapsto \alpha.$$
Clearly $\#\widetilde{\Psi}^{-1}(\alpha)=\nu(m,\alpha)$, so $\#\Psi^{-1}([\alpha])=\nu(m,\alpha)$ as well. It remains to prove the formulae for $\nu(m,\alpha)$.

Since $p\nmid m$, $E[m]\cong \bZ/m\bZ \times \bZ/m\bZ$. 
By the Chinese remainder theorem,
$$E[m]=\bigoplus_{\substack{\ell \mid m\\ \ell \text{ is prime }}}E[\ell]$$
and the map $$(C_{\ell})_{\substack{\ell \mid m\\ \ell \text{ is prime }}}\mapsto \bigoplus_{\substack{\ell \mid m\\ \ell \text{ is prime }}}C_{\ell}$$ is a bijection from $\prod_{\ell}\{\text{cyclic order-}\ell\text{ subgroup schemes of }E\}$ to $\{\text{cyclic order-}m\text{ subgroup schemes of }E\}$. Moreover, $C_m:=\bigoplus_{\substack{\ell \mid m\\ \ell \text{ is prime }}}C_{\ell}$ is stabilized by $\alpha$ if and only if each $C_{\ell}$ is stabilized by $\alpha$. This proves $$\nu(m,\alpha)=\prod_{\substack{\ell \mid m\\ \ell \text{ is prime }}}\nu(\ell,\alpha).$$

Finally, for a prime $\ell$, $\nu(\ell,\alpha)$ is precisely the number of $\alpha$-stable lines in $E[\ell]\cong \bZ/\ell\bZ \times \bZ/\ell\bZ$. Since $\alpha$ is trace-$0$, norm-$D$, its characteristic equation is $x^2=-D$. For $\ell\neq 2$, if $-D$ is a quadratic residue mod $\ell$, then $\alpha$ has two distinct eigenvalues in  $\bZ/\ell\bZ $, thus two eigenlines; if  $-D$ is not a quadratic residue mod $\ell$, then $\alpha$ does not have a stabilized line in $E[\ell]$. Thus $\nu(\ell,\alpha)=\left(1+\left(\frac{-D}{\ell}\right)\right)$ if $\ell$ is odd.

Now consider the case $\ell=2 \mid m$. Then $p\ne 2$ and $D$ is odd. The characteristic equation of $\alpha$ is equivalent to $x^2=1$ in $\bZ/2\bZ$, whose only root in $\bZ/2\bZ$ is $1$ with multiplicity $2$. Thus the eigenspace of $\alpha$ associated with the only eigenvalue $1$ is either the entire $E[2]$, or a line in $E[2]$. In the former case, $\alpha$ acts as the identity on $E[2]$ and stabilizes all three lines in $E[2]\cong \bZ/2\bZ\times \bZ/2\bZ$.
Thus $$\nu(2,\alpha)=\begin{cases}
    3\quad\text{ if }\alpha \text{ acts as identity on }E[2],\\
    1\quad\text{ otherwise.}
\end{cases}$$ 
\end{proof}

\begin{lemma} 
\label{lem:v-alpha-to-O} 
Following the notations and assumptions in Proposition~\ref{prop: drop-m-cyclic}, if $\cO_{\alpha'}=\cO_{\alpha}$, then $\nu(m,\alpha)=\nu(m,\alpha')$. In other words, $\nu(m,\cO_{\alpha}):=\nu(m,\alpha)$ is well-defined.
\end{lemma}
\begin{proof}
    It suffices to show $\nu(2,\alpha)$ only depends on $\cO_{\alpha}$, since $\nu(\ell,\alpha)$ for odd $\ell$ only depends on $D$ and $\ell$. Note that 
    \begin{align*}
        \alpha\text{ acts as identity on }E[2] \iff E[2]\subseteq \ker(\alpha-1) \iff \alpha-1=\beta\circ [2] \text{ for some }\beta\in \End(E),
    \end{align*}
    where $[2]$ denotes the multiplication by $2$ map, 
    so we obtain
    $$\nu(2,\alpha)=\begin{cases}
    3\quad\text{ if }\frac{\alpha-1}{2}\in R ,\\
    1\quad\text{ otherwise.}
\end{cases}$$
Meanwhile, 
$$\frac{\alpha-1}{2}\in R\iff \frac{\alpha-1}{2}\in \bQ(\alpha)\cap R\iff \frac{\sqrt{-D}-1}{2}\in \cO_{\alpha},$$
proving $\nu(2,\alpha)$ only depends on $\cO_{\alpha}$, as desired.
\end{proof}

\begin{remark}
    \label{rmk:v2-formula}
    Note that $\frac{\sqrt{-D}-1}{2}\in \cO_{\alpha}$ if and only if $D\equiv 3 \pmod{4}$ and $\cO_{\alpha}$ is the ring of integers $\cO_K$ of $K=\bQ(\sqrt{-D})$. Therefore, the above proof further gives
    $$\nu(2,\cO)=\begin{cases}
    3\quad\text{ if } D\equiv 3\pmod{4} \text{ and }\cO=\cO_K,\\
    1\quad\text{ otherwise.}
\end{cases}$$
\end{remark}

Meanwhile, suppose $\alpha'=u\alpha u^{-1}$ for some $u\in R^{\times}$, then $\iota_{\alpha'}(x)=u\iota_{\alpha}(x)u^{-1}$ for any $x\in \mathbb{Q}(\sqrt{-D})$ and thus $\cO_{\alpha'}=\iota_{\alpha'}^{-1}(\mathbb{Q}(\alpha')\cap R)=\iota_{\alpha'}^{-1}(\mathbb{Q}(u\alpha u^{-1})\cap R)=\iota_{\alpha'}^{-1}(u(\mathbb{Q}(\alpha)\cap R)u^{-1})=\iota_{\alpha}^{-1}(\mathbb{Q}(\alpha)\cap R)=\cO_{\alpha}$. Therefore, the order $\cO_{\alpha}$ is well-defined on $R^{\times}$ conjugacy classes, and we shall denote it by $\cO_{[\alpha]}$ to indicate this.

\begin{definition}
    Let $\cO$ be an order in $\mathbb{Q}(\sqrt{-D})$ containing $\sqrt{-D}$. We say $\cO$ is an order \emph{corresponding to} the $w_D$-fixed supersingular point $[(E,C_M)]$ if $\cO=\cO_{[\alpha]}$ and $[(E,C_M)]=\Phi([(\alpha,C_m)])$ for some $C_m$.  
\end{definition}

\begin{lemma}
\label{lem:cO-welldefine}
Let $p\ge 5$ be a prime and let  $D\ge 5, m\ge 1$ be square-free integers such that $(D,m)=1$ and $p\nmid m$. Let $M=Dm$. Given a $w_D$-fixed supersingular point $[(E,C_M)]$ on $X_0(M)_{\mathbb{F}_p}$, the corresponding order $\cO$ in $\bQ(\sqrt{-D})$ is uniquely determined (and we shall say \emph{the} corresponding order) if one of the following conditions holds:

    \begin{enumerate}
        \item $p\nmid D$; or
        \item $j(E)\ne 1728$; or
        \item $D\equiv 1, 2 \pmod 4$; or
        \item $m\ge 3$.
    \end{enumerate}
\end{lemma}

\begin{proof}
    By Proposition~\ref{prop:2-1-p-nmid-D} and \ref{prop:2-1-p-mid-D}, we may assume that $[(E,C_M)]=\Phi[(\alpha,C_m)]$ for some $(\alpha,C_m)$ in the domain of $\Phi$; moreover, if $[(E,C_M)]=\Phi[(\alpha,C_m)]=\Phi([(\beta,C_m')]$, we have $[(\beta,C_m')]=[(\alpha,C_m)]$ or $[(-\alpha,C_m)]$ unless we are in the following exceptional case: $p\mid D$ and $j(E)=1728$ and $\trd(g\alpha)=0$. 
    
    In the non-exceptional cases, it suffices to show $\cO_{[-\alpha]}=\cO_{[\alpha]}$. Let $\rho$ be  the nontrivial automorphism of $\bQ(\sqrt{-D})$. Then $\iota_{-\alpha}=\iota_{\alpha}\circ \rho$. Therefore, $\cO_{[-\alpha]}=\iota_{-\alpha}^{-1}(\bQ(-\alpha)\cap \End(E))=\iota_{-\alpha}^{-1}(\bQ(\alpha)\cap \End(E))=\rho^{-1}\iota_{\alpha}^{-1}(\bQ(\alpha)\cap \End(E))=\rho^{-1}\cO_{[\alpha]}=\cO_{[\alpha]}$, proving the uniqueness of the order $\cO$ corresponding to $[(E,C_M)]$ when we are not in the exceptional case.

    It remains to show why each of the conditions prevents the exceptional case or otherwise guarantee uniqueness of the corresponding $\cO$. Having (1) or (2) clearly prevents the exceptional case. Condition (3) implies that the only order containing $\sqrt{-D}$ in $\bQ(\sqrt{-D})$ is the ring of integers by Lemma~\ref{lem:norm_D_elt}, guaranteeing uniqueness.

    Condition (4) also prevents the exceptional case. Recall that $[(E,C_M)]=\Phi[(\alpha,C_m)]$, and we assume to the contrary that we have $m\ge 3$, $p\mid D$, $j(E)=1728$, and $\trd(g\alpha)=0$. Then by Proposition~\ref{prop:2-1-p-mid-D}, we obtain $[(E,C_M)]=\Phi[(\alpha,C_m)]=\Phi[(g\alpha, C_m)]$. Since $[(\alpha,C_m)]$ and $[(g\alpha, C_m)]$ are in the domain of $\Phi$, $\alpha (C_m)= C_m$ and $g\alpha(C_m)=C_m$. 
     Since $m\ge 3$ is square-free, $m$ has a prime factor $\ell\ge 3$. By the proof of Proposition~\ref{prop: drop-m-cyclic}, there is a cyclic subgroup scheme $C_{\ell}$ of order $\ell$ in $E[\ell]$ such that $\alpha (C_\ell)=C_\ell$ and $g\alpha(C_{\ell})=C_{\ell}$. Therefore, $g(C_{\ell})=g(\alpha(C_\ell))=C_{\ell}$. Under an identification $E[\ell]=\bZ/\ell\bZ\times \bZ/\ell\bZ$, $C_{\ell}$ is a line in the two dimensional vector space over $\bZ/\ell\bZ$. Let $P$ be a generator of $C_{\ell}$, then $\alpha P=aP$ and $gP=bP$ for some $a,b\in \bZ/\ell\bZ^{\times}$. Recall that $\trd(g\alpha)=0$ implies $g\alpha=-\alpha g$. So $baP=g\alpha P=-\alpha g P=-abP$. Thus $2ab=0$, contradicting $a,b\in \bZ/\ell\bZ^{\times}$ and $\ell$ is an odd prime. Therefore, Condition (4) also prevents the exceptional case.
\end{proof}

Note that an embedding $\iota: \mathbb{Q}(\sqrt{-D})\hookrightarrow B_{p,\infty}$ with $ \iota(\sqrt{-D})\in R$ is purely determined by its restriction on $\mathbb{Z}(\sqrt{-D})$, which has image contained in $R$. Therefore, $\{\iota: \mathbb{Q}(\sqrt{-D})\hookrightarrow B_{p,\infty}\mid \iota(\sqrt{-D})\in R\}$ is in bijection with $\{\iota: \mathbb{Z}(\sqrt{-D})\hookrightarrow R\}$, which decomposes into the disjoint union
$$\bigsqcup_{\cO\supseteq \mathbb{Z}[\sqrt{-D}]}\{\text{Optimal embeddings } \iota:\cO\hookrightarrow R\}$$ by \cite[(30.3.5)]{Voight}.

Correspondingly, we have the disjoint decompositions $$\left\{\alpha \in R\ : \trd(\alpha)=0,\nrd(\alpha)=D\right\}=\bigsqcup_{\cO\supseteq \mathbb{Z}[\sqrt{-D}]}\{\alpha\in R \mid \trd(\alpha)=0,\nrd(\alpha)=D,\text{ and } \cO_{[\alpha]}=\cO\},$$ 

$$\left\{(\alpha, C_m):\begin{array}{c}
   \alpha\in R, \trd(\alpha)=0,\nrd(\alpha)=D,\\ C_m \text{ an order-}m \text{ cyclic subgroup }\\ \text{scheme of } E,
    \alpha(C_m)=C_m 
\end{array}\right\}=\bigsqcup_{\cO\supseteq \mathbb{Z}[\sqrt{-D}]}\left\{(\alpha, C_m):\begin{array}{c}
   \alpha\in R, \trd(\alpha)=0,\nrd(\alpha)=D,\\ \cO_{[\alpha]}=\cO, C_m \text{ an order-}m\\ \text{ cyclic subgroup }\text{scheme of } E,\\
    \alpha(C_m)=C_m 
\end{array}\right\},$$

and, for a supersingular $E$, under at least one of the four conditions of well-definedness of the corresponding order in Lemma~\ref{lem:cO-welldefine},

$$\left\{\begin{array}{c}
     w_D\text{-fixed points associated}\\\text{ with } E 
     \text{ on }X_0(M)_{\overline{\mathbb{F}}_p}
\end{array}  \right\}=\bigsqcup_{\cO\supseteq \mathbb{Z}[\sqrt{-D}]}\left\{\begin{array}{c}
     w_D\text{-fixed points associated with } E \\
     \text{on }X_0(M)_{\overline{\mathbb{F}}_p}\text{ whose corresponding order is }\cO
\end{array}  \right\}.$$

Moreover, the disjoint components respect $R^{\times}=\End(E)^{\times}$ actions, and the $R^{\times}$ classes of the respective disjoint components are preserved under the map $\Phi$ in Proposition~\ref{prop:2-1-p-nmid-D} and \ref{prop:2-1-p-mid-D}, the map $\Psi$ in Proposition~\ref{prop: drop-m-cyclic}, and the bijection in Lemma~\ref{lem:bij-cm}. Hence by analyzing the fibers of these maps, we arrive at the following lemma.

\begin{lemma}
\label{lem:many-2-final}
Let $p\ge 5$ be a prime and let  $D\ge 5, m\ge 1$ be square-free integers such that $(D,m)=1$ and $p\nmid m$. Let $M=Dm$. Let $E$ be a supersingular elliptic curve over $\overline{\mathbb{F}}_p$ and let $R=\End(E)$. Let $\cO$ be an order in $\mathbb{Q}(\sqrt{-D})$ containing $\sqrt{-D}$. Assume one of the four conditions in Lemma~\ref{lem:cO-welldefine}, so that we may refer to \emph{the} corresponding order of a fixed point. Then 
    \begin{align*}
        &\#\{w_D\text{-fixed points associated with the elliptic curve }E\text{ on }X_0(M)\text{ whose corresponding order is }\cO\}\\
    =&\frac{1}{2}\nu(m,\cO)\#\left(\{\text{Optimal embeddings } \iota:\cO\hookrightarrow R\}/R^{\times}\right).
    \end{align*}
\end{lemma}

\begin{proposition}
\label{prop:count_fixing_order}
    Let $p\ge 5$ be a prime and let  $D\ge 5, m\ge 1$ be square-free integers such that $(D,m)=1$ and $p\nmid m$. Let $M=Dm$. Let $\cO$ be an order in $K=\mathbb{Q}(\sqrt{-D})$ containing $\sqrt{-D}$ and let $h(\cO)$ denote the class number of $\cO$. Assume one of the following conditions:
    \begin{enumerate}
        \item $p\nmid D$; or
        \item $p\equiv 1\pmod{4}$; or
        \item $D\equiv 1, 2 \pmod 4$; or
        \item $m\ge 3$.
    \end{enumerate}
    Then there are $$\frac{1}{2}h(\cO)\left(1-\left( \frac{-D}{p}\right)\right)\nu(m,\cO)$$ $w_D$-fixed supersingular points $[(E,C_{M})]$ in $X_0(M)_{\overline{\mathbb{F}}_p}$ such that $\cO$ is the order corresponding to $[(E, C_{M})]$.
\end{proposition}

\begin{proof}
 In this proof we will intensively use ideal classes and orders of quaternion algebras. While we cite the precise results we need, we refer the reader to Section 10.2 and Chapter 17 of \cite{Voight} for more detailed treatment of the topic. 

First note that the four conditions differ from those in Lemma~\ref{lem:cO-welldefine} only in (2). This is because we are no longer working with a fixed supersingular elliptic curve. By \cite[Chapter V Example 4.5]{Silverman}, $p\equiv 1\pmod{4}$ guarantees that an elliptic curve with $j$-invariant $1728$ is not a supersingular point. Hence Lemmas~\ref{lem:cO-welldefine} and \ref{lem:many-2-final} apply to any supersingular point $[(E,C_M)]$ relevant to this proposition.

    Let $E_0$ be a supersingular elliptic curve over $\overline{\mathbb{F}}_p$ and let $R_0=\End(E_0)$, a maximal order in $B_{p,\infty}$. By the Deuring correspondence \cite[Corollary 42.3.7]{Voight}, there is a bijection between isomorphism classes of supersingular elliptic curves over $\overline{\mathbb{F}}_p$ and the left class set    $\mathrm{Cls}_LR_0$, and if $E_{[I]}$ is the elliptic curve corresponding to $[I]$, then $\End(E_{[I]})\cong O_R(I)$, where $O_R(I)=\{\alpha\in B_{p,\infty}:I\alpha\subseteq I\}$ is the right order of $I$, and $\End(E_{[I]})^{\times}=\Aut(E_{[I]})\cong O_R(I)^{\times}$. 
    
    Let $n([I])$ denote the number of $w_D$-fixed points $[(E_{[I]},C_M)]$ corresponding to the order $\cO$. Let $\mu(\cO, O_R(I);O_R(I)^{\times})$ denote the number of $O_R(I)^{\times}$ conjugacy classes of optimal embeddings from $\cO$ to $O_R(I)$, or equivalently, the number of $\End(E_{[I]})^{\times}$ conjugacy classes of optimal embeddings from $\cO$ to $\End(E_{[I]})$. By Lemma~\ref{lem:many-2-final}, $n([I])=\frac{1}{2}\nu(m,\cO)\mu(\cO, O_R(I);O_R(I)^{\times})$.
    
    Therefore, the total number of $w_D$-fixed supersingular points in $X_0(M)_{\overline{\mathbb{F}}_p}$ corresponding to the order $\cO$ is given by

    \begin{equation}
        \sum_{[I]\in\mathrm{Cls}_LR_0} n([I]) =\frac{1}{2} \nu(m,\cO)\sum_{[I]\in\mathrm{Cls}_LR_0} \mu(\cO, O_R(I);O_R(I)^{\times}).
    \end{equation}

It now remains to show 
\begin{equation}
\label{eqn:optimal-embedding-count}
\sum_{[I]\in\mathrm{Cls}_LR_0} \mu(\cO, O_R(I);O_R(I)^{\times}) = h(\cO)\left(1-\left( \frac{-D}{p}\right)\right).
\end{equation}

By the left-right switch via the canonical involution \cite[Section 17.1, bottom of p. 277]{Voight} and \cite[Theorem 30.7.3]{Voight} applied to $R_0$, we have \begin{equation}
\sum_{[I]\in\mathrm{Cls}_LR_0} \mu(\cO, O_R(I);O_R(I)^{\times}) = h(\cO)\mu(\cO_p,R_{0,p},R_{0,p}^{\times}).
\end{equation} 

By \cite[Proposition 30.5.3(b)]{Voight}, we have $\mu(\cO_p,R_{0,p},R_{0,p}^{\times})=\left(1-\left( \frac{-D}{p}\right)\right)$, so we obtain \eqref{eqn:optimal-embedding-count} as desired.
\end{proof}

\begin{remark}
    Note again that we only need any one of the four conditions in Proposition~\ref{prop:count_fixing_order} for the count to hold. In particular, if we are in the $p\nmid D$ case, the first condition is satisfied and the count holds. It is only in the $p\mid D$ case that we need to look for one of (2), or (3), or (4) as a sufficient (not claiming to be necessary) condition for the count to hold.
\end{remark}

\begin{proposition}
\label{prop:fixed-point-formula}
    Let $p\ge 5$ be a prime and let  $D\ge 5, m\ge 1$ be square-free integers such that $(D,m)=1$. Let $F(D,m,p)$ be the number of $w_D$-fixed supersingular points on $X_0(Dm)_{\overline{\bF}_p}$. Then
    $$F(D,m,p)=\frac{v(D)}{2}\left(1-\left(\frac{-D}{p}\right)\right)\prod_{\substack{\ell \mid m\\ \ell \text{ an odd prime }\\\ell\neq p}}\left(1+\left(\frac{-D}{\ell}\right)\right),$$
    where $$v(D)=\begin{cases}
    h(-4D) &\text{ for } D\equiv 1,2 \pmod{4},\\
    h(-D)+h(-4D)&\text{ for } D\equiv 3 \pmod{4},
\end{cases}$$
if $D\equiv 1, 2\pmod{4}$ or $m$ is odd, and in the case where $p\mid D$, $D\equiv 3 \pmod{4}$, and $m=1$, we additionally assume $p\equiv 1\pmod 4$.
\end{proposition}

\begin{proof}
Assume that $D\equiv 1, 2\pmod{4}$ or $m$ is odd. Since $(D,m)=1$, we have three cases: $p\nmid D$ and $p\nmid m$; $p\nmid D$ and $p\mid m$; $p\mid D$ and $p\nmid m$.

\underline{Case 1:} Assume $p\nmid D$ and $p\nmid m$. Then Condition (1) of Proposition~\ref{prop:count_fixing_order} is met and the result applies. We get 
\begin{equation}
\label{eqn:F-initial}
F(D,m,p)=\sum_{\cO\supseteq \mathbb{Z}[\sqrt{-D}]}\frac{1}{2}h(\cO)\left(1-\left( \frac{-D}{p}\right)\right)\nu(m,\cO).
\end{equation}
Moreover, by Lemma~\ref{lem:v-alpha-to-O}, Proposition~\ref{prop: drop-m-cyclic}, and Remark~\ref{rmk:v2-formula}, if $D\equiv 1, 2\pmod{4}$ or $m$ is odd, then the $\nu(2,\cO)$ term either equals $1$ or does not appear as a factor in the formula. Since $p$ is not a factor of $m$, we obtain $$\nu(m,\cO)=\prod_{\substack{\ell \mid m\\ \ell \text{ an odd prime }\\\ell\neq p}}\left(1+\left(\frac{-D}{\ell}\right)\right).$$  Combining with \eqref{eqn:F-initial}, we get $$F(D,m,p)=\frac{1}{2}\left(\sum_{\cO\supseteq \mathbb{Z}[\sqrt{-D}]}h(\cO)\right)\left(1-\left( \frac{-D}{p}\right)\right)\prod_{\substack{\ell \mid m\\ \ell \text{ an odd prime }\\\ell\neq p}}\left(1+\left(\frac{-D}{\ell}\right)\right),$$ which by Lemma~\ref{lem:norm_D_elt} is precisely the proposed formula.

\underline{Case 2:} Assume $p\nmid D$ and $p\mid m$. Then $m=pm'$ and $Dm=p(Dm')$ for some integer $m'\ge 1$. By Lemma~\ref{lem:M_vs_pM} and Case 1, we have $$F(D,m,p)=F(D,m',p)=\frac{v(D)}{2}\left(1-\left(\frac{-D}{p}\right)\right)\prod_{\substack{\ell \mid m'\\ \ell \text{ an odd prime }\\\ell\neq p}}\left(1+\left(\frac{-D}{\ell}\right)\right).$$ But the odd prime factors of $m'$ that are not equal to $p$ are exactly the same as the odd prime factors of $m$ not equal to $p$, so we again obtain the proposed formula.

\underline{Case 3:} Now assume $p\mid D$ and $p\nmid m$. Recall also that we are assuming $D\equiv 1, 2\pmod{4}$ or $m$ is odd. If $D\equiv 1, 2\pmod{4}$, then Condition (3) of Proposition~\ref{prop:count_fixing_order} is met. If $m$ is odd and $m\ne 1$, then we have Condition (4) of Proposition~\ref{prop:count_fixing_order}. If neither is true, we are in the case where $p\mid D$, $D\equiv 3 \pmod{4}$, and $m=1$, where we additionally assume $p\equiv 1\pmod{4}$, meeting Condition (2) of Proposition~\ref{prop:count_fixing_order}. In any case, we obtain the per-order count in Proposition~\ref{prop: drop-m-cyclic} and the claimed formula follows as in Case 1.
\end{proof}

\begin{remark}
\label{rmk:Ogg}
In \cite[Page 454]{Ogg}, the number of fixed points of the $w_{D}$ action on $X_0(Dm)$ when $(D,m)=1$ is given by (after changing $N'$ to $D$, $N''$ to $m$ from his notation) $$v(D)\prod_{\ell\mid m}\left(1+\left(\frac{-4D}{\ell}\right)\right),$$
which looks different from our formula
$$F(D,m,p)=\frac{v(D)}{2}\left(1-\left(\frac{-D}{p}\right)\right)\prod_{\substack{\ell \mid m\\ \ell \text{ an odd prime }\\\ell\neq p}}\left(1+\left(\frac{-D}{\ell}\right)\right),$$
in a few ways. We shall explain why there is no contradiction.

First, under our assumptions, the primes $\ell$ showing up in the Legendre symbols $\left(\frac{-D}{\ell}\right)$ are odd, so we do have $\left(\frac{-D}{\ell}\right)=\left(\frac{-4D}{\ell}\right)$. Then the only difference is the term $\frac{1}{2}\left(1-\left(\frac{-D}{p}\right)\right)$ in our formula, which replaces $\left(1+\left(\frac{-D}{p}\right)\right)$ in Ogg's formula if $p\mid m$ and replaces $1$ (not a factor) in Ogg's formula if $p\nmid m$.

This is because Ogg's setting was in characteristic zero, while we are counting \emph{supersingular} fixed points \emph{in characteristic} $p$. In our case, for a $w_D$-fixed supersingular point to exist on $X_0(Dm)_{\overline{\bF}_p}$, we need $\bQ(\sqrt{-D})\hookrightarrow B_{p,\infty}$, so we always need the term $\frac{1}{2}\left(1-\left(\frac{-D}{p}\right)\right)$ as an indicator. (This is more subtle in the case $p\mid D$ where we added assumptions to make sure the $2$-to-$1$ identifications hold and we are still left with an integer.) When $p\mid m$, the existence of $C_p$ in Ogg's case requires $-D$ to be a quadratic residue modulo $p$, in which case there are two $\sqrt{-D}$ eigenlines on $E[p]$; whereas in our case of a supersingular elliptic curve in characteristic $p$, we always have a unique $C_p$ arising as the kernel of the relative Frobenius, explaining the term $\left(1+\left(\frac{-D}{p}\right)\right)$ in Ogg's formula that we do not have.  
\end{remark}

\section{The quotient $X_0^+(p)$ of $X_0(p)$}
\label{sec:X0+p}
We start with the Deligne--Rapoport model of $X_0(p)$ over $\overline{\mathbb{F}}_p$, which is two irreducible components $X_0(1)$ intersecting at supersingular points. It is well-known that $w_p$ interchanges the two components.  [See Figure~\ref{fig:X0p} below for an illustration.]

\begin{figure}[h]
\includegraphics[scale=0.5]{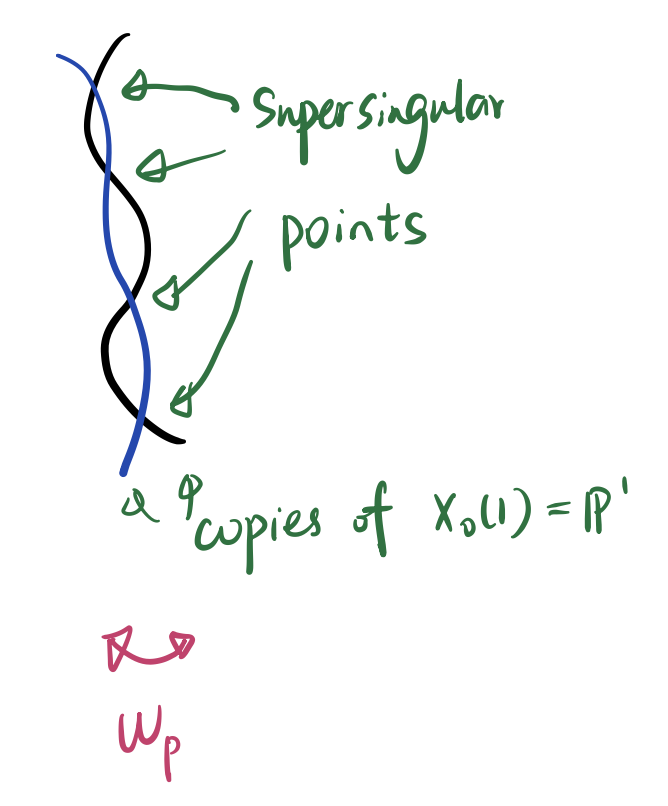}
\caption{$X_0(p)$ over $\overline{\mathbb{F}}_p$}
\label{fig:X0p}
\end{figure}

To see this picture and the $w_p$ action from the moduli perspective, recall that a geometric noncuspidal point on $X_0(p)$ corresponds to an isomorphism class of $p$-cyclic isogenies between elliptic curves and $w_p$ sends an isogeny to its dual. By \cite[Theorem 13.3.3]{Katz--Mazur} we have that for ordinary elliptic curves $E_1, E_0$, a $p$-isogeny 
\begin{equation*}
    E_0\xrightarrow{\pi_{0,1}} E_1
\end{equation*}
 factors as either Type-$(1,0)$,

\begin{equation}
\label{eqn:E_Frob}
E_0\xrightarrow{F_0}E_0^{(p)}\xrightarrow{\cong}E_1\xrightarrow{id}E_1
\end{equation}

\noindent or Type-$(0,1)$,

\begin{equation}
\label{eqn:E_Vers}
E_0\xrightarrow{id}E_0\xrightarrow{\cong}E_1^{(p)}\xrightarrow{V_1}E_1.
\end{equation}
Since such a $p$-isogeny is either type $(1,0)$ or type $(0,1)$, this implies that a $p$-isogeny between two ordinary elliptic curves is either equal to an isomorphism composed with the Frobenius or equal to the Verschiebung composed with an isomorphism. The two decomposition types correspond to the two irreducible components in the picture above. A $p$ isogeny between supersingular elliptic curves has both decompositions, and that is why the intersections are at supersingular points. When a cyclic isogeny has Type-$(0,1)$ (resp. (1,0)), its dual isogeny has Type-$(1,0)$ (resp. (0,1)). Therefore $w_p$ swaps the two irreducible components.

Next, we analyze how $w_p$ acts on the supersingular points. Recall that when an elliptic curve $E_0$ is supersingular we have $w_p([E_0])=[E_0^{(p)}]$. We also always have (regardless of supersingular or ordinary) $w_p([x: E_0\to E_1])=([\hat{x}:E_1\to E_0])$, where $\hat{x}$ denotes the dual isogeny of $x$.

If $x:E_0 \to E_1$ is an isogeny between two supersingular elliptic curves, then we have the following two cases.

\begin{enumerate}
    \item If $j(E_0)\in \mathbb{F}_p$, then we have that $\sigma_p(j(E_0))=j(E_0)$ where $\sigma_p:x\mapsto x^p$ fixes $\mathbb{F}_p$. Then $j(E_1)=j(E_0^{(p)})=\sigma_p(j(E_0))=j(E_0)$ and $E_1\cong E_0$. In this case, the point $[x:E_0\to E_1]$ is mapped to itself under $w_p$.

    \item If $j(E_0) \not\in \mathbb{F}_p$, then it is in $\mathbb{F}_{p^2}$ (\cite{Deuring1941}, as cited in \cite[V.3.1]{Silverman}, and in this case we still have $E_1\cong E_0^{(p)}$ but this time in particular $E_0^{(p)}$ not isomorphic to $E_0$ and $w_p$ takes $[x:E_0 \to E_1]$ to another point $[\hat{x}:E_1 \to E_0]$.
\end{enumerate}

\begin{remark}
\label{rmk:Fp2-Fp}
    The following is another way to think of the two cases. We know that we always have $w_p([x:E_0\to E_1])=[\hat{x}:E_1\to E_0]$ and $w_p([\hat{x}:E_1\to E_0])=[x:E_0\to E_1]$, so forgetting the level structure we have $w_p([E_0])=[E_1]$ and $w_p([E_1])=[E_0]$. When $E_0,E_1$ are supersingular, we further have $E_0^{(p)}\cong E_1$, so $[E_1]=[E_0^{(p)}]$. We obtain $w_p([E_0])=[E_0^{(p)}]$ and $w_p([E_0^{(p)}])=[E_0]$. Let $\sigma_p:\overline{\bF}_p\to \overline{\bF}_p, \sigma\mapsto\sigma^p$. Then $j([E_1])=j(w_p([E_0]))=j([E_0^{(p)}])=\sigma_p j([E_0])$ and similarly $j([E_0])=j(w_p([E_1]))=\sigma_p j([E_1])=\sigma_p^2 j([E_0])$. Therefore, $j(E_0)$ is either fixed by $\sigma_p$, or not by $\sigma_p$ but by $\sigma_p^2$. In the former case, $j(E_0)=j(E_1)\in\bF_p$; in the latter case, $j(E_0)\ne j(E_1)\in \bF_{p^2}\backslash\bF_{p}$ and $j(E_0)$ and $j(E_1)$ are Galois conjugates of each other.
\end{remark}

\begin{theorem}
\label{thm:X0pmodwp}  The following are equal.
\begin{enumerate}
    \item[(a)] The genus of $X_0^+(p)/\mathbb{Q}$. 
    
    \item[(b)] The number of singular points on $X_0^+(p)_{\overline{\bF}_p}$ (model obtained by quotienting the Deligne--Rapoport model by $w_p$).
    
    \item[(c)] Half of the number of supersingular points on $X_0(p)$ over $\overline{\mathbb{F}}_p$ \textbf{not} fixed by $w_p$. 

    \item[(d)] Half of the number of supersingular $j$ invariants (all in $\mathbb{F}_{p^2}$) that are not in $\mathbb{F}_p$.
\end{enumerate}

When $p\ge 5$, it is also known that the total number $S_{p^2}$ of supersingular $j$-invariants (in $\mathbb{F}_{p^2}$) is given by 

\begin{equation*}
  S_{p^2}=  \left\lfloor\frac{p}{12}\right\rfloor+\begin{cases}
        0\quad\text{ if } p\equiv 1 \mod {12}\\
        1\quad\text{ if } p\equiv 5,7 \mod {12}\\
        2\quad\text{ if } p\equiv 11 \mod {12}
    \end{cases}
\end{equation*}
(\cite[Theorem V.4.1(c)]{Silverman}) and that the total number $S_p$ of supersingular $j$-invariants in $\mathbb{F}_p$ is given by
\begin{equation*}
   S_p= \begin{cases}
        \frac{1}{2}h(-4p)\quad\text{ if } p\equiv 1 \mod 4\\
        h(-p)\quad\text{ if } p\equiv 7 \mod 8\\
        2h(-p)\quad\text{ if } p\equiv 3 \mod 8,
    \end{cases}
\end{equation*}
 (\cite[Theorem 14.18]{Cox89}, as cited in \cite[p.2]{Delfs--Galbraith}) where $h(d)$ denotes the class number of the imaginary quadratic order of discriminant $d$. In particular, Item~(d) above is explicit when $p\ge 5$.
\end{theorem}

\begin{remark}
    In \cite[p.2]{Delfs--Galbraith}, the authors introduced $h(d)$ as the class number of the imaginary quadratic field $\mathbb{Q}(\sqrt{d})$, which agrees with the class number of the imaginary quadratic order of discriminant $d$ when $d$ is a fundamental discriminant. All of the inputs of $h$ above are fundamental discriminants, so the two notions agree here. 
    
    Here and throughout, we use $h(d)$ to denote the class number of the imaginary quadratic order of discriminant $d$, and later on we apply this notion to non-fundamental discriminants as well.
\end{remark}

\begin{proof}
We first show $(a)=(b)$. Let $r$ denote the number of singular points on $X_0^+(p)_{\overline{\bF}_p}$.  Since the special fiber of $X_0^+(p)_{\overline{\bF}_p}$ is $\mathbb{P}^1$ intersecting itself at $r$ nodal singularities, we obtain the arithmetic genus of $X_0^+(p)_{\overline{\bF}_p}$ is $g(\mathbb{P}^1)+r=r$ by Proposition~\ref{prop:genus-dual-graph}. By Remark~\ref{rmk:arith-geom-genus}, we obtain the arithmetic genus of $X_0^+(p)/\mathbb{Q}$ is $r$, and since $X_0^+(p)/\mathbb{Q}$ is smooth, its geometric genus (or now unambiguously, genus) is also $r$.

    Next, we show $(b)=(c)$. From Proposition~\ref{prop:quot_sing}, the only singularities on the quotient of $X_0(p)_{\overline{\bF}_p}$ by $w_p$ are obtained from singular points on $X_0(p)_{\overline{\bF}_p}$ not fixed by $w_p$, and each pair $(x,w_px)$ results in one nodal singularity on the quotient. Since the singular points on $X_0(p)_{\overline{\bF}_p}$ are precisely the supersingular points, we obtain the result.

    Finally, we show $(c)=(d)$. First note from the Deligne--Rapoport model and Lemma~\ref{lem:M_vs_pM} that the number of supersingular points on $X_0(p)$ over $\overline{\mathbb{F}}_p$ is the same as the number of supersingular points of $X_0(1)_{\overline{\bF}_p}$, which corresponds to supersingular elliptic curves without level structure. Therefore, we can simply count the supersingular $j$-invariants. By the analysis in Remark 4.1 and above it, we demostrated that the supersingular points not fixed by $w_p$ are precisely those with $j$ invariant in $\mathbb{F}_{p^2}\backslash \mathbb{F}_p$. This proves $(c)=(d)$. 
\end{proof}
\begin{figure}[h]
    \includegraphics[scale=0.18]{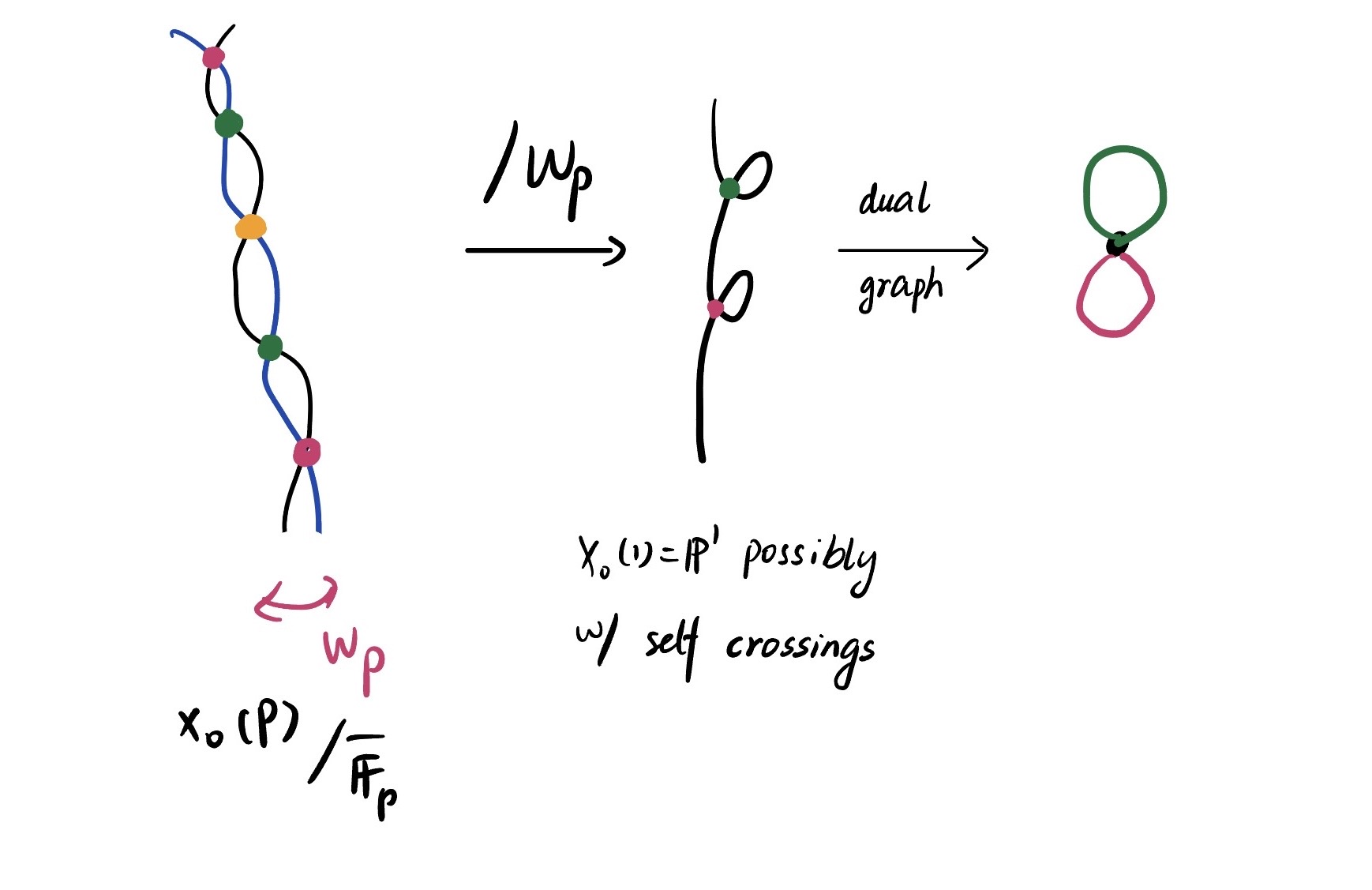}
    \caption{Illustration of the quotient}
\end{figure}

\section{Quotients of $X_0(pq)$}
\label{sec:X0*pq}
Similar to the $X_0(p)$ case, the Deligne--Rapoport model of $X_0(pq)$ over $\overline{\mathbb{F}}_p$ is two irreducible components $X_0(q)$ intersecting at supersingular points. In this case, $w_p$ also interchanges the two components. [See Figure~\ref{fig:X0pq} below for an illustration.] To see the moduli interpretation of the two irreducible components in this case, note that any $pq$-cyclic isogeny $E_0\xrightarrow{\pi_{0,1}} E_1$ factors either as 

\begin{equation}
\label{eqn:E_Frob-qlevel}
E_0\xrightarrow{F_0}E_0^{(p)}\xrightarrow{\cong}E_1\xrightarrow{\varphi}E_1
\end{equation}
for some cyclic $q$-isogeny $\varphi$, or 

\begin{equation}
\label{eqn:E_Versqlevel}
E_0\xrightarrow{\psi}E_0\xrightarrow{\cong}E_1^{(p)}\xrightarrow{V_1}E_1
\end{equation}
for some cyclic $q$-isogeny $\psi$. This is similar to what is in \eqref{eqn:E_Frob} and \eqref{eqn:E_Vers}, replacing the identity map with a cyclic $q$-isogeny.

\begin{figure}[h]
\includegraphics[scale=0.5]{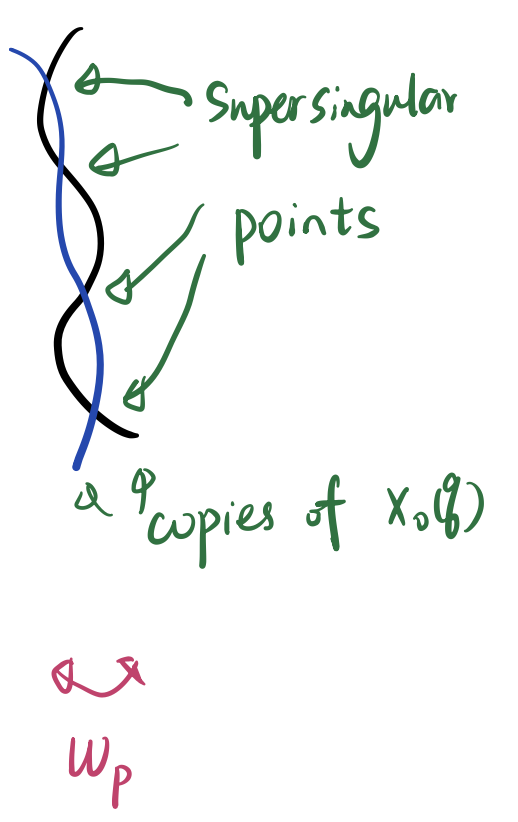}
\caption{$X_0(pq)$ over $\overline{\mathbb{F}}_p$}
\label{fig:X0pq}
\end{figure}

We study the quotient of this model by $\langle w_p,w_q\rangle$, and we obtain this quotient in two ways, first quotienting by $w_p$ then by $w_q$, or first by $w_q$ then by $w_p$, as illustrated in Figure~\ref{fig:two-routes}.

\begin{figure}
\includegraphics[scale=0.4]{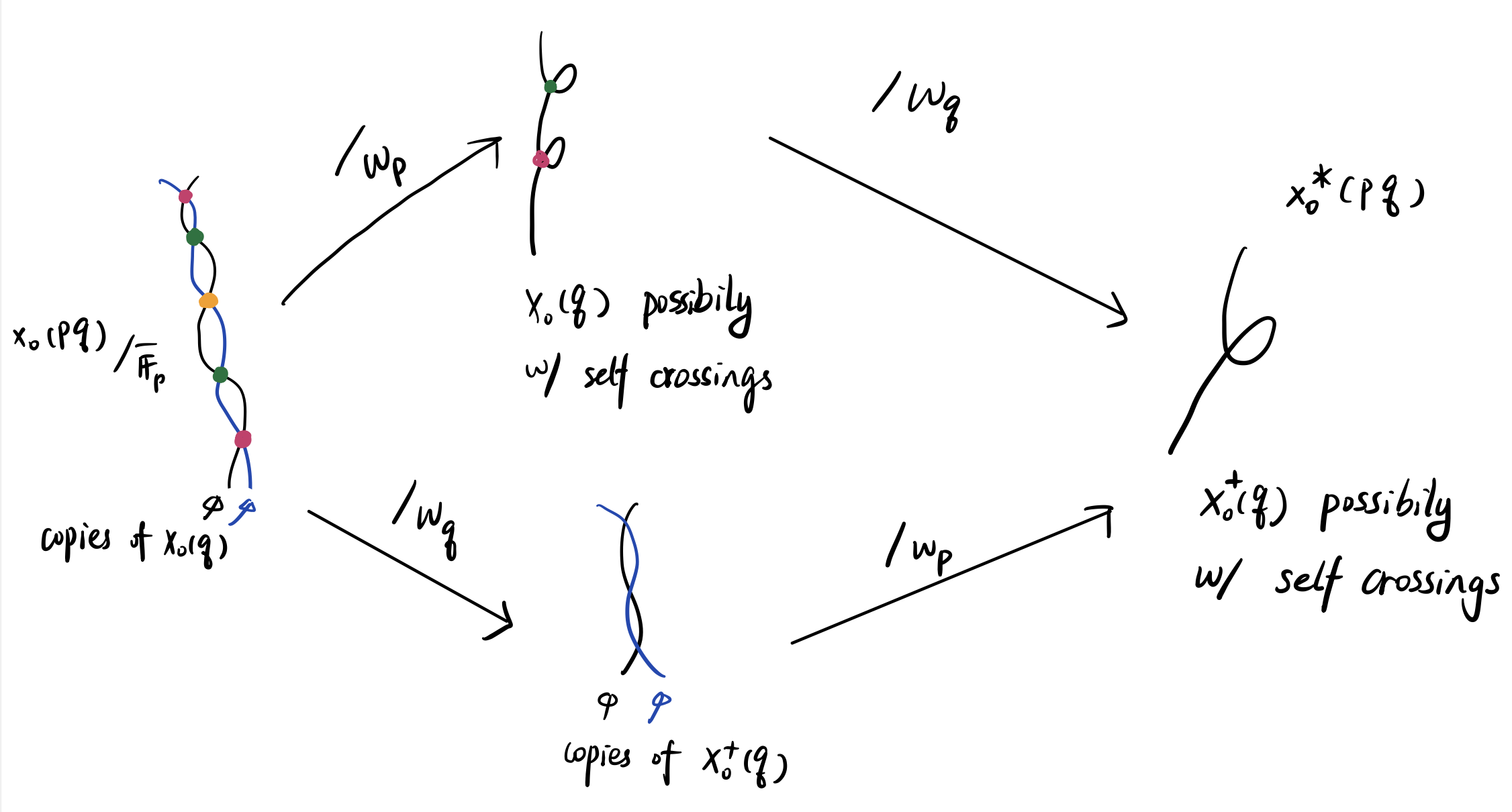}
\caption{Two routes}
\label{fig:two-routes}
\end{figure}

\subsection{First layer quotients}
We first focus on the step from $X_0(pq)$ to $X_0(pq)/w_p$. Recall from Theorem~\ref{thm:X0pmodwp} that $S_{p^2}$ denotes the number of supersingular $j$-invariants over $\overline{\mathbb{F}}_p$ (all of which are in $\mathbb{F}_{p^2}$), $S_{p}$ denotes the number of supersingular $j$-invariants in $\mathbb{F}_p$, and there are explicit formulae for $S_{p^2}$ and $S_p$ when $p>3$. Let $$SS_{p,q}:=\left\{(E,C_q) :\begin{array}{c} E \text{ a supersingular elliptic curve over } \overline{\mathbb{F}}_p,\\ \text{ and } C_q \text{ a cyclic subgroup of order }q \text{ in } E\end{array}\right\}/{\cong},$$ which is the set of supersingular points on $X_0(q)_{\overline{\bF}_p}$, in bijection (via $[(E,C_q)]\mapsto [(E,C_q+C_p)]$) with the set of supersingular points on $X_0(pq)_{\overline{\bF}_p}$. (c.f. Lemma~\ref{lem:M_vs_pM}.)

\begin{proposition}
\label{prop:X0pq}
Let $p$ be a prime number. The genus of $X_0(pq)$ is $2g\left(X_0(q)\right)+(\#SS_{p,q})-1$. Moreover, when $p\equiv 1 \pmod{12}$, $\#SS_{p,q}=(q+1)\lfloor\frac{p}{12}\rfloor$.
\end{proposition}

\begin{proof}
Recall that $X_0(pq)_{\overline{\bF}_p}$ consists of two irreducible components $X_0(q)_{\overline{\bF}_p}$ intersecting at $\#SS_{p,q}$ nodal crossings, by Proposition~\ref{prop:genus-dual-graph} and Remark~\ref{rmk:arith-geom-genus}, we obtain $g\left(X_0(pq)\right)=2g\left(X_0(q)\right)+(\#SS_{p,q})-1$.

Now assume $p\equiv 1 \pmod {12}$. For each isomorphism class of supersingular elliptic curve $E$, there are $q+1$ non-isomorphic cyclic subgroups of order $q$. Meanwhile, $S_{p^2}=\lfloor\frac{p}{12}\rfloor$ when $p\equiv 1 \pmod{12}$ (c.f. Theorem~\ref{thm:X0pmodwp}). So $\#SS_{p,q}=(q+1)\lfloor\frac{p}{12}\rfloor$ when $p\equiv 1 \pmod{12}$.
\end{proof}
\begin{remark}\label{rmk:1mod12}
    The condition $p\equiv 1\pmod{12}$ is to avoid elliptic curves with extra automorphisms (those with $j=0,1728$). In general, we can still decompose $SS_{p,q}$ by the isomorphism classes $[E]$ and count the number of cyclic subgroups of $E[q]$ not identified by $\Aut(E)$, but the $\Aut(E)$ orbits need not be singletons. In the case where  $j=1728$, Arpin in \cite[Proposition 3.6]{Arpin} summarizes $\Aut(E)$ actions on the order-$q$ cyclic subgroups, noting that a similar procedure works for $j=0$ and that the condition on $q$ being prime is only for simplicity. In these cases, one should simply replace the factor $(q+1)$ with the number of $\Aut(E)$-orbits of order-$q$ cyclic subgroups.
\end{remark}

\begin{proposition}
\label{prop:X0pqmodwp}
    Let $p$ be a prime number such that $p\equiv 1 \pmod{12}$. The model of $X_0(pq)/w_p$ over $\overline{\bF}_p$ obtained from taking the quotient of the Deligne--Rapoport model of $X_0(pq)_{\overline{\bF}_p}$ by $w_p$ is given by $X_0(q)_{\overline{\bF}_p}$ with $\frac{1}{2}S_{p^2}(q+1)-\frac{1}{2}eS_p$ self-crossings that are ordinary double points, where $S_{p^2}=\lfloor\frac{p}{12}\rfloor$, $S_p=\frac{h(-4p)}{2}$ (c.f. Theorem~\ref{thm:X0pmodwp}), 
    and $e$ is the number of eigenlines of the Frobenius endomorphism on a supersingular elliptic curve $E/\bF_{p}$ acting on $E[q]$, and is given by $e=\left(1+\left(\frac{-p}{q}\right)\right)$ if $q$ is odd.
    
    Consequently, the genus of $X_0(pq)/w_p$ is equal to 

$$g(X_0(q))+\frac{1}{2}S_{p^2}(q+1)-\frac{1}{2}eS_p.$$
\end{proposition}

\begin{proof}
Since $w_p$ identifies the two irreducible components $X_0(q)$ in $X_0(pq)_{\overline{\bF}_p}$, $X_0(pq)/w_p$ over $\overline{\mathbb{F}}_p$ is $X_0(q)$ with possible self-crossings. By Proposition~\ref{prop:quot_sing}, the self-crossings are ordinary double points. Let $C$ denote the number of self-crossings, then by Proposition~\ref{prop:genus-dual-graph} and Remark~\ref{rmk:arith-geom-genus}, $g(X_0(pq)/w_p)=g(X_0(q))+C$. Moreover, the formulas for $S_{p^2}$ and $S_p$ when $p\equiv 1\pmod{12}$ follow directly from Theorem~\ref{thm:X0pmodwp}.  It remains to show that $$C=\frac{1}{2}S_{p^2}(q+1)-\frac{1}{2}eS_p$$ and that $e=\left(1+\left(\frac{-p}{q}\right)\right)$ if $q$ is odd.

   Note that the self-crossings of  $X_0(pq)_{\overline{\mathbb{F}}_p}/w_p$ arise from pairs of distinct supersingular points on $X_0(pq)$ identified by $w_p$, since all $w_p$-fixed points are branch-swapping in the sense of Proposition~\ref{prop:quot_sing}. Thus $2C=\#\{[(E,C_q)]\in SS_{p,q}: w_p[(E,C_q+C_p)]\neq [(E,C_q+C_p)]\}.$

    Given $[(E,C_q)]\in SS_{p,q}$,  by Remark~\ref{rmk:Fp2-Fp} and the analysis above it, $E$ is not fixed by $w_p$ if and only if $j(E)\in \mathbb{F}_{p^2}\backslash \mathbb{F}_{p}$. Since there are $S_{p^2}-S_p$ supersingular $j$-invariants in $\mathbb{F}_{p^2}\backslash \mathbb{F}_{p}$, and each supersingular elliptic curve has $q+1$ cyclic subgroups of order $q$, we obtain
$$\#\{[(E,C_q)]\in SS_{p,q}: w_p([E])\neq [E]\}=(S_{p^2}-S_p)(q+1).$$

       Now consider $[(E,C_q)]\in SS_{p,q}$ such that $j(E)\in \mathbb{F}_p$. In this case $E$ is defined over $\bF_p$. Fix an $\bF_p$ model of $E$ and thus an identification of $E^{(p)}$ with $E$, then the $p$-th power Frobenius defines an endomorphism $\Frob_p:E\to E$ and $w_p[(E,C_q+C_p)]=[(E,\Frob_pC_q+C_p)]$ \cite[V.1.17]{Deligne--Rapoport}. Since $p\equiv 1\pmod{12}$, $E$ does not have extra automorphisms other than $\pm 1$, so $[(E,\Frob_pC_q+C_p)]=[(E,C_q+C_p)]$ if and only if $\Frob_pC_q=C_q$, i.e. $C_q$ is a $\Frob_p$ eigenline in $E[q]$. Therefore,
$$\#\{[(E,C_q)]\in SS_{p,q}: w_p([E])= [E]\text{ and } w_p[(E,C_q+C_p)]\neq [(E,C_q+C_p)] \}=S_p(q+1-e).$$

    Since $\{[(E,C_q)]\in SS_{p,q}: w_p[(E,C_q+C_p)]\neq [(E,C_q+C_p)]\}$ is the disjoint union of the above two subsets, we obtain $$2C=(S_{p^2}-S_p)(q+1)+S_p(q+1-e)=S_{p^2}(q+1)-eS_p,$$ as desired.

    Now we compute $e$. Recall that $\Frob_{p}$ acts on $E[q]\cong \bZ/q\bZ\times \bZ/q\bZ$ with the characteristic polynomial $x^2-a_{p}x+p$, where $a_{p}=0$ if $E$ is supersingular \cite[V.2.1 and proof of V.4.1]{Silverman}. When $q$ is odd, $\Frob_p$ has two distinct eigenvalues in $\bZ/q\bZ$ thus two eigenlines if $-p$ is a quadratic residue modulo $q$, and no eigenvalue/eigenline otherwise. Therefore, $e=\left(1+\left(\frac{-p}{q}\right)\right)$ if $q$ is odd.
\end{proof}

\begin{proposition}
\label{prop:X0pqmodwq}
    Let $p,q$ be distinct primes $\ge 5$. The model of $X_0(pq)/w_q$ over $\overline{\bF}_p$ obtained from taking the quotient of the Deligne--Rapoport model of $X_0(pq)_{\overline{\bF}_p}$ by $w_q$ is given by two irreducible components  $X_0^+(q)_{\overline{\bF}_p}$ intersecting at $\frac{g(X_0(pq))+1-2g(X_0(q))-F(q,p,p)}2+F(q,p,p)$ ordinary double points, where $F(q,p,p)$ is the number of fixed supersingular points of the $w_q$-action on $X_0(pq)_{\overline{\mathbb{F}}_p}$ (c.f. Proposition~\ref{prop:fixed-point-formula})  and is given by 
$$ F(q,p,p)=\frac{v(q)}{2} \left(1-\left( \frac{-q}{p}\right)\right),$$
where $$v(q)=\begin{cases}
    h(-4q) &\text{ for } q\equiv 1,2 \pmod{4},\\
    h(-q)+h(-4q)&\text{ for } q\equiv 3 \pmod{4}.
\end{cases}$$

    Consequently, the genus of $X_0(pq)/w_q$ is equal to 

$$\frac{g(X_0(pq))+1-2g(X_0(q))-F(q,p,p)}2+F(q,p,p) -1 +2g(X_0^+(q)).$$
\end{proposition}

\begin{proof} Recall that $X_0(pq)_{\overline{\mathbb{F}}_p}$ is two irreducible components $X_0(q)$ intersecting at supersingular points with nodal singularity, say there are $N$ such points. By Proposition~\ref{prop:quot_sing} applied to the $w_q$ action on $X_0(pq)$ and the fact that $w_q$ acts on the two $X_0(q)$ branches in a branch-preserving way, we obtain that $X_0(pq)/w_q$ is two irreducible components $X_0^+(q)$ intersecting at $\frac{N-F(q,p,p)}{2}+F(q,p,p)$ ordinary singularities. 

By Proposition~\ref{prop:genus-dual-graph} and Remark~\ref{rmk:arith-geom-genus}, we have
\begin{equation}
\label{eqn:g-pq-vs-q}
    g(X_0(pq))=N-1+2g(X_0(q))
\end{equation}
and 
\begin{equation}
\label{eqn:g-pqmodq-vs-q+}
    g(X_0(pq)/w_q)=\frac{N-F(q,p,p)}{2}+F(q,p,p)-1+2g(X_0^+(q)).
\end{equation}

From \eqref{eqn:g-pq-vs-q}, we get $N=g(X_0(pq))+1-2g(X_0(q))$, and substituting it into \eqref{eqn:g-pqmodq-vs-q+}, we obtain the claimed genus formula for $X_0(pq)/w_q$. The formula for $F(q,p,p)$ follows directly from Proposition~\ref{prop:fixed-point-formula} and the fact that $p,q$ are distinct primes $\ge 5$.   
\end{proof}

\subsection{Second layer quotients}
In this section, we describe the quotient $X_0^*(pq)=X_0(pq)/\langle w_p,w_q\rangle$ in two ways, following the two routes, first quotient by $w_p$ and then quotient by $w_q$, or first quotient by $w_q$ and then by $w_p$.

\begin{theorem}
\label{thm:X0pqmodwpwq}
Let $p,q$ be distinct primes $\geq 5$. The model of $X_0^*(pq)$ over $\overline{\bF}_p$ obtained from taking the quotient of the Deligne--Rapoport model of $X_0(pq)_{\overline{\bF}_p}$ by $\langle w_p,w_q\rangle$ is given by $X_0^+(q)_{\overline{\bF}_p}$ with
\begin{equation}
    \label{eqn:two-crossings}
    \begin{split}
&\frac12\left((g(X_0(pq)/w_p)-g(X_0(q)))-\frac12F(pq,1,p)+\frac12F(q,p,p)\right)\\
   =& \frac12\left((g(X_0(pq)/w_q)+1-2g(X_0^+(q)))-\frac12F(pq,1,p)-\frac12F(p,q,p)\right)
    \end{split}
\end{equation}
self-crossings that are ordinary double points, where $F(q,p,p)$ is the number of fixed supersingular points of the $w_q$ action on $X_0(pq)_{\overline{\mathbb{F}}_p}$ and is given by \[F(q,p,p)=\frac{v(q)}{2}\left(1-\left(\frac{-q}{p}\right)\right),\]
and $F(pq,1,p)$ is the number of fixed supersingular points of the $w_{pq}$ action on $X_0(pq)_{\overline{\mathbb{F}}_p}$ and is given by \[ F(pq,1,p)=\frac{v(pq)}{2},\] if $pq\equiv 1 \pmod{4}$ or $p\equiv 1\pmod 4$, and $F(p,q,p)$ is the number of $w_p$-fixed supersingular points on $X_0(pq)_{\overline{\bF}_p}$ and is given by $$F(p,q,p)=\frac{v(p)}{2}\left(1+\left(\frac{-p}{q}\right)\right)$$ where $$v(x)=\begin{cases}
    h(-4x) &\text{ for } x\equiv 1,2 \pmod{4},\\
    h(-x)+h(-4x)&\text{ for } x\equiv 3 \pmod{4}.
\end{cases}$$

Consequently, the genus of $X_0^*(pq)$ is equal to 
\begin{align*}
&\frac12\left((g(X_0(pq)/w_p)-g(X_0(q)))-\frac12F(pq,1,p)+\frac12F(q,p,p)\right)+g(X_0^+(q))\\
=&\frac12\left((g(X_0(pq)/w_q)+1-2g(X_0^+(q)))-\frac12F(pq,1,p)-\frac12F(p,q,p)\right)+g(X_0^+(q)).
\end{align*}
\end{theorem}

\begin{proof}
    The formulae for $F(q,p,p),F(pq,1,p),F(p,q,p)$ under the specified conditions for $F(pq,1,p)$ follow directly from Proposition~\ref{prop:fixed-point-formula}. It suffices to show that the quotient of the Deligne--Rapoport model of $X_0(pq)_{\overline{\bF}_p}$ by $\langle w_p, w_q\rangle$ is indeed $X_0^+(q)$ with the claimed number of self-crossings. The genus formulae will then follow by Proposition~\ref{prop:genus-dual-graph} and Remark~\ref{rmk:arith-geom-genus}. Note that $X_0^*(pq)=X_0(pq)/\langle w_p, w_q\rangle=(X_0(pq)/w_p)/w_q=(X_0(pq)/w_q)/w_p$. The two proposed formulae for the number of self-crossings are obtained by the two routes: (1) first quotient by $w_p$ then by $w_q$, or (2) first quotient by $w_q$, then by $w_p$.

    \underline{Route (1):} Recall that after quotienting the Deligne--Rapoport model of $X_0(pq)_{\overline{\bF}_p}$ by $w_p$, we obtain $X_0(q)_{\overline{\bF}_p}$ with $g(X_0(pq)/w_p)-g(X_0(q))$ self-crossings, when further quotienting by $w_q$, we obtain $X_0^+(q)_{\overline{\bF}_p}$ with self-crossings, which we count according to Proposition~\ref{prop:quot_sing}.
    
    The singular points of $X_0(pq)_{\overline{\bF}_p}/w_p$ are equivalence classes $[u]=\{u, w_p(u)\}$, where $u$ is a supersingular point on $X_0(pq)_{\overline{\bF}_p}$ such that $u\ne w_p(u)$. For such a $[u]$ to be fixed by $w_q$ there are two possibilities, either $u$ is fixed by $w_q$, in which case $w_p(u)$ is fixed by $w_q$ as well (the equivalence class is fixed pointwise), or $w_p(u)=w_q(u)$, in which case $w_q(w_p(u))=u$ (the equivalence class is fixed but not pointwise). Note that $w_p(u)=w_q(u)$ if and only if $w_{pq}(u)=u$, so we will use this convenient reformulation. In the sense of Proposition \ref{prop:quot_sing}, the pointwise fixed case is an example of a branch-preserving action on a fixed point and the non-pointwise fixed case is an example of a branch-swapping action on a fixed point. 

    Since $u$ is fixed by $w_q$ (resp. $w_{pq}$) if and only if $w_p(u)$ is fixed by $w_q$ (resp. $w_{pq}$), the number of point-wise (resp. non-point-wise) fixed classes $[u]$ under $w_q$ action is half of the number of supersingular points fixed by $w_q$ (resp. $w_{pq}$) on $X_0(pq)_{\overline{\bF}_p}$, namely $\frac{1}{2}F(q,p,p)$ (resp. $\frac{1}{2}F(pq,1,p)$). After removing these fixed crossings, the rest are pairwise identified by the involution $w_q$ on $X_0(pq)_{\overline{\bF}_p}/w_p$, and all such identified pairs are crossings in $(X_0(pq)_{\overline{\bF}_p}/w_p)/w_q$ by Proposition~\ref{prop:quot_sing}(2), contributing $\frac12\left((g(X_0(pq)/w_p)-g(X_0(q)))-\frac12F(pq,1,p)-\frac12F(q,p,p)\right)$ crossings on $(X_0(pq)_{\overline{\bF}_p}/w_p)/w_q$. By the second part of Proposition~\ref{prop:quot_sing}(2), the branch-preserving fixed crossing points of $w_q$ on $X_0(pq)_{\overline{\bF}_p}/w_p$ also remain a crossing point, contributing $\frac{1}{2}F(q,p,p)$ crossings on $(X_0(pq)_{\overline{\bF}_p}/w_p)/w_q$. The other fixed crossing points of $w_q$ action on $X_0(pq)/w_p$ are branch-swapping, which by Proposition~\ref{prop:quot_sing}(3) becomes smooth on $(X_0(pq)_{\overline{\bF}_p}/w_p)/w_q$. This proves $(X_0(pq)_{\overline{\bF}_p}/w_p)/w_q$ is $X_0^+(q)$ with $$\frac12\left((g(X_0(pq)/w_p)-g(X_0(q)))-\frac12F(pq,1,p)-\frac12F(q,p,p)\right)+\frac{1}{2}F(q,p,p)$$ self-crossings, as desired.

    \underline{Route (2):} Recall that after quotienting the Deligne--Rapoport model of $X_0(pq)_{\overline{\bF}_p}$ by $w_q$, we obtain two irreducible components  $X_0^+(q)_{\overline{\bF}_p}$ intersecting at $(g(X_0(pq)/w_q)+1-2g(X_0^+(q)))$ ordinary double points, 
    which are equivalence classes $[u]=\{u, w_q(u)\}$, where $u,w_q(u)$ may or may not be distinct. (Recall from the proof of Proposition~\ref{prop:X0pqmodwq} that crossings in $X_0(pq)_{\overline{\bF}_p}/w_q$ may come from $w_q$ fixed points or $w_q$ identified pairs.) For such a $[u]$ to be fixed by $w_p$ there are three possibilities: 

    \begin{enumerate}
        \item[(i)] $u=w_q(u)$ and $w_p(u)=u$. So the class $[u]=\{u,w_q(u)\}$ is a singleton set and $w_p$ fixes it.
        \item[(ii)] $u\ne w_q(u)$ and $w_p(u)=u$. In this case we automatically get $w_p(w_q(u))=w_q(u)$, so that the class $[u]=\{u, w_q(u)\}$ contains two distinct elements that are fixed pointwise by $w_p$. 
        \item[(iii)] $u\ne w_q(u)$ and $w_p(u)=w_q(u)$. In this case $w_p(w_q(u))=u$, so that the class $[u]=\{u, w_q(u)\}$ contains two distinct elements that are swapped by $w_p$, keeping the class fixed  by $w_p$. Note that $w_p(u)=w_q(u)\iff w_p(w_q(u))=u\iff w_{pq}(u)=u$.  
    \end{enumerate}

All three cases are branch-swapping fixed points in the sense of Proposition \ref{prop:quot_sing}, because the action of $w_p$ is to identify the two irreducible components  $X_0^+(q)_{\overline{\bF}_p}$. 

Now we will count the $w_p$-fixed classes $[u]=\{u,w_q(u)\}$. Let $n$ denote the number of common fixed supersingular points by $w_p$ and by $w_q$ on $X_0(pq)_{\overline{\bF}_p}$. Note that if a supersingular point $u$ on $X_0(pq)_{\overline{\bF}_p}$ is fixed by any two of $w_p, w_q, w_{pq}$, it is fixed by the third, so equivalently, $n$ is also the count of common fixed supersingular points by $w_q$ and $w_{pq}$. Note also that $u$ is fixed by $w_p$ (resp. $w_{pq}$) if and only if $w_p(u)$ is fixed by $w_q$ (resp. $w_{pq}$). 

Type (i) and (ii) fixed classes consists of $w_p$-fixed supersingular points on $X_0(pq)_{\overline{\bF}_p}$. Among the $F(p,q,p)$ such points, $n$ of them form the type (i) singleton fixed classes, while $(F(p,q,p)-n)$ of them form type (ii) pointwise fixed classes with two distinct elements, forming $\frac12(F(p,q,p)-n)$ type (ii) classes. Type (iii) fixed classes consists of $w_{pq}$-fixed supersingular points on $X_0(pq)_{\overline{\bF}_p}$ that are not fixed by $w_q$. There are $F(pq,1,p)-n$ such points, forming $\frac12\left(F(pq,1,p)-n\right)$ type (iii) classes.

Therefore, there are in total $$n+\frac12(F(p,q,p)-n)+\frac12\left(F(pq,1,p)-n\right)=\frac12 F(p,q,p)+\frac12 F(pq,1,p)$$ crossing points on $X_0(pq)_{\overline{\bF}_p}/w_q$ that are fixed by $w_p$, all in a branch-swapping way.
 By Proposition~\ref{prop:quot_sing}(3), these fixed points become smooth on $(X_0(pq)_{\overline{\bF}_p}/w_q)/w_p$, while the rest of the crossings on $X_0(pq)_{\overline{\bF}_p}/w_q$ are identified pairwise by the involution $w_p$, contributing $$\frac12\left((g(X_0(pq)/w_q)+1-2g(X_0^+(q)))-\frac12F(pq,1,p)-\frac12F(p,q,p)\right)$$ ordinary double points by Proposition~\ref{prop:quot_sing}(2). 
\end{proof}

\begin{remark}
    While the detailed analysis of the quotients in Route (2) involves the count of supersingular points on $X_0(pq)_{\overline{\bF}_p}$ fixed by both $w_p$ and by $w_q$, it cancels out in the eventual description of our model. This is a prototype for what we observe in Proposition~\ref{prop:fixed-classes}.
\end{remark}

\begin{corollary}
    Let $p,q$ be distinct primes $\ge 5$, then $$g(X_0(pq)/w_p)=\frac{g(X_0(pq))-F(p,q,p)+1}{2}.$$ 
\end{corollary}
\begin{proof}
    Simplifying \eqref{eqn:two-crossings}, we obtain
\begin{equation}
    \label{eqn:two-crossings-v2}
    g(X_0(pq)/w_p)-g(X_0(q))+\frac{1}{2}F(q,p,p)=g(X_0(pq)/w_q)-2g(X_0^+(q))-\frac{1}{2}F(p,q,p)+1.
\end{equation}
By Proposition~\ref{prop:X0pqmodwq}, we have
    \begin{align*}
&g(X_0(pq)/w_q)-2g(X_0^+(q))-\frac{1}{2}F(p,q,p)+1\\
=&\frac{g(X_0(pq))+1-2g(X_0(q))-F(q,p,p)}2+F(q,p,p)-\frac{1}{2}F(p,q,p)\\
=&\frac{g(X_0(pq))-F(p,q,p)+1}{2}-g(X_0(q))+\frac{1}{2}F(q,p,p).
\end{align*}
Comparing with \eqref{eqn:two-crossings-v2} we obtain the desired result. 
\end{proof}
\begin{remark}
    Note that the results used to obtain the above Corollary did not use any additional assumption other than $p,q$ are distinct primes $\ge 5$. If we further assume $p\equiv 1\pmod{12}$, we may plug in the formula for $g(X_0(pq)/w_p)$ from Proposition~\ref{prop:X0pqmodwp}, the formula for $g(X_0(pq))$ from Proposition~\ref{prop:X0pq}, as well as the $S_{p^2}, S_p$ formulae recalled in Theorem~\ref{thm:X0pmodwp} and the $F(p,q,p)$ formula recalled in Theorem~\ref{thm:X0pqmodwpwq}, to further verify that both sides of the above Corollary is given by $$g(X_0(q))+\frac{1}{2}\left\lfloor\frac{p}{12}\right\rfloor(q+1)-\frac{1}{4}h(-4p)\left(1+\left(\frac{-p}{q}\right)\right).$$
\end{remark}

\subsection{A concrete example}
 When $p,q$ are distinct prime divisors of $N$, the magma commands \verb|X0NQuotient(N,[])|, \verb|X0NQuotient(N,[p])|, and 
 \verb|X0NQuotient(N,[p,q])| return the curves $X_0(N)$, $X_0(N)/w_p$, and $X_0(N)/\langle w_p,w_q\rangle$ respectively. The function \verb|Genus| can be used to compute the genus of these curves.

\begin{example} 
    Consider the case where $p=13$, $q=7$, and $N=pq=91$. We have $13\equiv 1\pmod {12}$ and $S_{13^2}=\lfloor \frac{13}{12}\rfloor=1$. So $\#SS_{13,7}=1(7+1)=8$. Since the genus of $X_0(7)$ is $0$, the genus of $X_0(91)$, by Proposition~\ref{prop:X0pq}, is $7$.

    Now, by Proposition \ref{prop:X0pqmodwp}, to find the genus of $X_0(91)/w_{13}$, we compute $S_{13}=\frac{1}{2}h(-4*13)=1$.  So the only supersingular $j$-invariant mod $13$ is in $\mathbb{F}_{13}$, and the corresponding elliptic curve class is fixed by $w_{13}$, but it has $8$ non-isomorphic cyclic subgroups of order $7$, among which $e=\left(1+\left(\frac{-13}{7}\right)\right)=2$ are fixed by $\Frob_{13}$ and hence giving rise to fixed points of $w_{13}$.  Therefore, $1(8-2)=6$ of the supersingular points are not fixed by $w_{13}$ and are identified with each other, creating $\frac{1}{2}\times 6=3$ crossings. 
Therefore, $X_0(91)_{\overline{\mathbb{F}}_{7}}/w_{13}$ is $X_0(7)$ (which has genus $0$) with $3$ nodal self-crossings, and hence have genus $3$.

Now, by Proposition~\ref{prop:X0pqmodwq}, to find the genus of $X_0(91)/w_7$, we first compute the number of fixed supersingular points of $w_7$ on $X_0(91)_{\overline{\bF}_{13}}$, which is $F(7,13,13)=\frac{1}{2}\left(h(-7)+h(-28)\right)\left(1-\left(\frac{-7}{13}\right)\right)=\frac{1}{2}(1+1)(1-(-1))=2$. 
So among the eight crossings of $X_0(91)_{\overline{\bF}_{13}}$, $2$ are fixed by $w_7$ while $6$ are pairwise identified by the involution $w_7$,  so after quotienting $w_7$ we are left with 2 crossings from the fixed supersingular points and $6/2=3$ crossings from the identified pairs, giving a total of 5 crossings. Since $X_0(91)_{\overline{\bF}_{13}}/w_{7}$ is two irreducible components $X_0^+(7)_{\overline{\bF}_{13}}$ (which has genus~$0$) intersecting at $5$ ordinary crossings, we obtain $g(X_0(91)/w_7)=4$.

    Now we would like to get $g(X_0^*(91))$ in two ways as in Theorem~\ref{thm:X0pqmodwpwq}. First let's consider the route obtained by quotienting $X_0(91)_{\overline{\bF}_{13}}/w_{13}$ by $w_7$.  As computed earlier, there are $F(7,13,13)=2$ supersingular points on $X_0(91)_{\overline{\bF}_{13}}$ fixed by $w_7$, corresponding to $\frac12 F(7,13,13)=1$ crossing $[u]=\{u,w_{13}(u)\}$ on $X_0(91)_{\overline{\bF}_{13}}/w_{13}$ fixed by $w_7$ in branch-preserving way. Meanwhile, there are $F(91,1,13)=\frac12(h(-91)+h(-4*91))=\frac12 (2+6)=\frac82=4$ supersingular points on $X_0(91)_{\overline{\bF}_{13}}$ fixed by $w_{91}$, corresponding to $\frac12 F(91,1,13)=2$ crossings $[u]=\{u,w_{13}(u)\}$ on $X_0(91)_{\overline{\bF}_{13}}/w_{13}$ fixed by $w_7$ in branch-swapping way. Now that all $3$ crossings on $X_0(91)_{\overline{\bF}_{13}}/w_{13}$ are fixed by $w_7$ and only one of them is in branch-preserving way, we obtain $(X_0(91)_{\overline{\bF}_{13}}/w_{13})/w_7$ is $X_0^+(7)_{\overline{\bF}_{13}}$ (which has genus~$0$) with $1$ self-crossing, showing $g(X^*_0(91))=1$.

Now we consider the route obtained by quotienting $X_0(91)_{\overline{\bF}_{13}}/w_{7}$ by $w_{13}$. Recall that $w_{13}$ has $2$ fixed supersingular points on $X_0(91)_{\overline{\bF}_{13}}$, which can also be confirmed by the formula $F(13,7,13)=\frac12 h(-4*13)\left(1+\left(\frac{-13}{7}\right)\right)=\frac12 * 2*2=2$. Recall also that $F(91,1,13)=4$. By our analysis in Route (2) of the proof of Theorem~\ref{thm:X0pqmodwpwq}, $\frac12 F(13,7,13)+\frac12 F(91,1,13)=1+2=3$ of the crossings on $X_0(91)_{\overline{\bF}_{13}}/w_{7}$ are fixed by $w_{13}$, all of which in a branch-swapping way. Since the original $X_0(91)_{\overline{\bF}_{13}}/{w_7}$ had 5 crossings, we are left with 2 crossings that get identified by the involution $w_13$, becoming one ordinary double point in the quotient. We obtain again that $(X_0(91)_{\overline{\bF}_{13}}/w_{13})/w_7$ is $X_0^+(7)_{\overline{\bF}_{13}}$ (which has genus~$0$) with $1$ self-crossing, showing $g(X^*_0(91))=1$. 
\end{example}

\section{Atkin--Lehner quotients of $X_0(N)$ for square-free $N$}
\label{sec:square-free}
Throughout this section, let $N$ be a square-free positive integer. We discuss the Deligne--Rapoport model of $X_0(N)_{\overline{\mathbb{F}}_p}$ and its quotients by Atkin--Lehner involutions for some prime $p$ that divides $N$. Now that we fix $N$ and $p$, for the rest of this section, we shall use the shorthand notation $$F(D):=F(D,N/D,p)$$ for the number of $w_D$-fixed supersingular points on $X_0(N)_{\overline{\bF}_p}$.

\begin{definition}
\label{def:star}
Let $d_1$ and $d_2$ be square-free integers. We define the product $*$ as 
$d_1*d_2=d_1d_2/\gcd{(d_1,d_2)^2}.$ For example, if $p_1,\dotsc, p_5$ are distinct primes and $d_1=p_1p_2p_3$, $d_2=p_1p_3p_4p_5$, then $d_1d_2=p_1^2p_3^2p_2p_4p_5$ and $d_1*d_2=p_2p_4p_5$.
\end{definition}

\begin{proposition}
Let $N=p_1\dots p_r$ be a square-free positive integer. Suppose that one of these primes is $1\pmod{12}$ (see Remark \ref{rmk:1mod12} for how to remove this condition). Without loss of generality, suppose this prime is $p_r$. Let $N'=N/p_r$. Then $X_0(N)_{\overline{\mathbb{F}}_{p_r}}$ consists of two irreducible components $X_0(N/p_r)_{\overline{\bF}_{p_r}}$ intersecting at $\left(\prod_{j=1}^{r-1}(p_j+1)\right)S_{p_r^2}$\footnote{See Theorem \ref{thm:X0pmodwp} for the definition of $S_{p^2}$.} ordinary double points. Moreover, the genus of $X_0(N)$ is \[2g(X_0(N'))+\left(\prod_{j=1}^{r-1}(p_j+1)\right)S_{p_r^2}-1.\]
\end{proposition}

\begin{proposition}
\label{prop:fixed-classes}
    Let $N$ be a square-free positive integer divisible by a prime $p\ge 5$, and let $d_1,\dotsc,d_k, d_{k+1}$ be positive divisors of $N$ such that $w_{d_1},\dotsc, w_{d_k}$ is a minimal set of generators for $G:=\langle w_{d_1},\dotsc, w_{d_k} \rangle$ and $w_{d_{k+1}}\notin G$. The number of supersingular points on $X_0(N)_{\overline{\bF}_p}/G$ fixed by $w_{d_{k+1}}$ is 
    $$\frac{1}{2^k}\sum\limits_{d\in \mathcal{D}} F(d_{k+1}*d) ,$$
    where $\mathcal{D}$ is the set consisting of all elements of $\{1, d_1, \dots, d_k\}$ and all their $*$-products; equivalently $\cD$ is the set such that $G=\{w_d:d\in \cD\}$.
\end{proposition}

\begin{proof}
For each $d\in \cD$, let $X_d$ denote the set of supersingular points on $X_0(N)_{\overline{\bF}_p}$ fixed by $w_{d_{k+1}}w_d$. Let $X=\cup_{d\in \cD}X_d$. The group $G$ acts on $X$ via its action on $X_0(N)_{\overline{\bF}_p}$. Let $|X/G|$ denote the number of orbits of $G$ action on $X$, and for any $g\in G$, let $X^g:=\{x\in X : g(x)=x\}$. By Burnside's Lemma, $$|X/G|=\frac{1}{|G|}\sum_{g\in G}|X^g|=\frac{1}{2^k}\sum_{g\in G}|X^g|=\frac{1}{2^k}\left\vert\{(g,x):g\in G, x\in X,  g(x)=x\}\right\vert.$$ It now suffices to show that
\begin{enumerate}
    \item The number of supersingular points on $X_0(N)_{\overline{\bF}_p}/G$ fixed by $w_{d_{k+1}}$ is $|X/G|$; and
    \item $\sum\limits_{d\in \mathcal{D}} F(d_{k+1}*d)=\left\vert\{(g,x):g\in G, x\in X, g(x)=x\}\right\vert$.
\end{enumerate}

For (1), 
we argue that the set of orbits of $X$ under $G$ is in bijection with the set of fixed points of $X_0(N)_{\overline{\bF}_p}/G$ fixed by $w_{d_{k+1}}$.

Define \[\varphi: X/G=\{orb_G(x) | x\in X\}\to \{[x]\in X_0(N)_{\overline{\bF}_p}/G\mid w_{d_{k+1}}[x]=[x]\}, \quad orb_G(x)\mapsto [x].\]

Suppose that $orb_G(x)=orb_G(y)$ then there exists $w_d\in G$ such that $y=w_dx$ but then $[x]=[y]$ so $\varphi$ is well-defined. Now suppose $[x]=[y]$, then $y=w_d(x)$ for some $w_d\in G$, but then $orb_G(x)=orb_G(y)$, so $\varphi$ is injective. Now suppose that $[x]$ is an arbitrary element of $X_0(N)_{\overline{\bF}_p}/G$ fixed by $w_{d_{k+1}}$, then $w_{d_{k+1}}[x]=[x]$, equivalently $[w_{d_{k+1}}(x)]=[x]$ which happens if and only if $w_{d_{k+1}}(x)=w_d(x)$ for some $w_d\in G$, which is equivalent to saying that $w_{d_{k+1}}w_d(x)=x$, or in other words, $x\in X$. Thus $\varphi$ is surjective.

For (2), note that \begin{align*}\sum\limits_{d\in \mathcal{D}} F(d_{k+1}*d)&=\sum\limits_d |X_d|=|\{(d,x)\mid d \in\mathcal{D}, x\in X_d\}|\\& =|\{(w_d,x)\mid w_d \in G, x\in X_d\}|\\& =|\{(w_d,x)\mid w_d \in G, x\in X, w_{d_{k+1}}w_d(x)=x\}|\\&=\bigsqcup\limits_{x\in X}|\{(w_d,x)\mid w_d \in G, w_{d_{k+1}}w_d(x)=x\}|\\&=\sum\limits_{x\in X}|\{w_d\in G\mid w_{d_{k+1}}w_d(x)=x\}|.
\end{align*}

Note that $\sum\limits_{x\in X}|\{w_d\in G\mid w_{d_{k+1}}w_d(x)=x\}|=\sum\limits_{x\in X}|\{w_d\in G\mid w_d(x)=x\}|$. Indeed, for each $x\in X$, there is a $w_{d_0}\in G$ such that $w_{d_{k+1}}w_{d_0}(x)=x$, and the map $\{w_d\in G\mid w_{d_{k+1}}w_d(x)=x\}\to \{w_d\in G\mid w_d(x)=x\},  w_d\mapsto w_{d_0}w_d$ is a bijection.

Finally, $\sum\limits_{x\in X}|\{w_d\in G\mid w_d(x)=x\}|=\left\vert\{(g,x):g\in G, x\in X, g(x)=x\}\right\vert$, concluding the proof of (2).
\end{proof}

We study Atkin--Lehner quotients of the Deligne--Rapoport model of $X_0(N)_{\overline{\bF}_p}$ inductively as follows.

\begin{proposition}
\label{prop:sqfr}
Let $N$ be a square-free positive integer divisible by some prime $p\geq 5$. Let $d_1, \dotsc, d_k, d_{k+1}$ be positive divisors of $N$ such that $w_{d_1},\dotsc ,w_{d_{k+1}}$ is a minimal set of generators for $\langle w_{d_1},\dotsc ,w_{d_{k+1}}\rangle$. Without loss of generality, we assume $p\nmid d_i$ for each $i=1,\dotsc, k$. Then $X_0(N)_{\overline{\bF}_p}/\langle w_{d_1},\dotsc ,w_{d_k}\rangle$ obtained from taking the quotient of the Deligne--Rapoport model of $X_0(N)_{\overline{\bF}_p}$ is two irreducible components $X_0(N/p)/\langle w_{d_1},\dotsc ,w_{d_k}\rangle$ intersecting at ordinary double points. Let $C_k$ denote the number of such ordinary double points. Let $\cD_k$ be the set consisting of all elements of $\{1, d_1, \dots, d_k\}$ and all their $*$-products.

If $p\mid d_{k+1}$, then $X_0(N)_{\overline{\bF}_p}/\langle w_{d_1},\dotsc ,w_{d_{k+1}}\rangle$ is the curve $X_0(N/p)_{\overline{\bF}_p}/\langle w_{d_1},\dots, w_{d_k}\rangle$ with $$\frac12{\left( C_k-\frac{1}{2^k}\left(\sum\limits_{d\in \cD_k} F(d_{k+1}*d) \right)\right)}$$ self-crossings that are ordinary double points, and consequently
the genus of $X_0(N)/\langle w_{d_1}, \dots, w_{d_{k+1}}\rangle$ is 
\[g(X_0(N/p)/\langle w_{d_1},\dots, w_{d_k}\rangle)+ \frac12{\left( C_k-\frac{1}{2^k}\left(\sum\limits_{d\in \cD_k} F(d_{k+1}*d) \right)\right)}.\]

If $p\nmid d_{k+1}$, then $X_0(N)_{\overline{\bF}_p}/\langle w_{d_1},\dotsc ,w_{d_{k+1}}\rangle$ is two irreducible components $X_0(N/p)_{\overline{\bF}_p}/\langle w_{d_1}, \dots, w_{d_{k+1}}\rangle$ intersecting at 
$$\frac12{\left( C_k+\frac{1}{2^k}\left(\sum\limits_{d\in \cD_k} F(d_{k+1}*d) \right)\right)}$$
ordinary double points, and consequently the genus of $X_0(N)/\langle w_{d_1}, \dots, w_{d_k+1}\rangle$ is 
\begin{align*}
&2g(X_0(N/p)/\langle w_{d_1}, \dots, w_{d_{k+1}}\rangle)+\frac12{\left( C_k+\frac{1}{2^k}\left(\sum\limits_{d\in \cD_k} F(d_{k+1}*d) \right)\right)} -1.
\end{align*}
\end{proposition}

\begin{proof}
By Proposition \ref{prop:fixed-classes}, the number of fixed supersingular points of $X_0(N)_{\overline{\mathbb{F}}_p}/\langle w_{d_1},\dots, w_{d_k}\rangle$ by $w_{d_{k+1}}$ is $$\frac{1}{2^k}\left(\sum\limits_{d\in \cD_k} F(d_{k+1}*d) \right), \text{ which we denote by }n\text{ for the rest of the proof}.$$

If $p\mid d_{k+1}$, then $w_{d_{k+1}}$ interchanges the two irreducible components $X_0(N/p)/\langle w_{d_1},\dotsc ,w_{d_k}\rangle$ of  $X_0(N)/\langle w_{d_1},\dotsc ,w_{d_k}\rangle$ and thus acts on any fixed point in a branch-swapping way. Consequently, after quotienting by $w_{d_{k+1}}$ we obtain one irreducible component $X_0(N/p)/\langle w_{d_1},\dotsc ,w_{d_k}\rangle$ with self-crossings. By Proposition~\ref{prop:quot_sing}, the $w_{d_{k+1}}$-fixed ordinary double points become smooth after the quotient, while the rest of the ordinary double points on $X_0(N)/\langle w_{d_1},\dotsc ,w_{d_k}\rangle$ are pairwise identified, leading to $\frac{1}{2}(C_k-n)$ ordinary double points, as claimed.

If $p\nmid d_{k+1}$, then then $w_{d_{k+1}}$ preserves each of the two irreducible components $X_0(N/p)/\langle w_{d_1},\dotsc ,w_{d_k}\rangle$ of  $X_0(N)/\langle w_{d_1},\dotsc ,w_{d_k}\rangle$ and thus acts on any fixed point in a branch-preserving way. Consequently, after quotienting by $w_{d_{k+1}}$ we obtain two irreducible components $X_0(N/p)/\langle w_{d_1},\dotsc ,w_{d_{k+1}}\rangle$. By Proposition~\ref{prop:quot_sing}, the $w_{d_{k+1}}$-fixed ordinary double points remain ordinary double points after the quotient, while the rest of the ordinary double points on $X_0(N)/\langle w_{d_1},\dotsc ,w_{d_k}\rangle$ are pairwise identified, leading to $\frac{1}{2}(C_k-n)+n$ ordinary double points, as claimed.

The genus formulae follow from Proposition~\ref{prop:genus-dual-graph} and Remark~\ref{rmk:arith-geom-genus}.
\end{proof}

Let $N=p_1\dots p_r$ be a square-free positive integer and let $p=p_r$. Let $W_N$ denote the full Atkin--Lehner group $\langle w_{p_1},\dotsc, w_{p_r}\rangle$. Then any element of $W_N$ can be written as $w_d$ for some $d\mid N$. Consider the indicator homomorphism $i_p:W_N\to \bZ/2\bZ$ defined by $$i_p(w_d)=\begin{cases}
    0,\quad\text{ if }p\nmid d \\
    1, \quad\text{ if }p\mid d.
\end{cases}$$ Let $H$ be a subgroup of $W_N$, then $i_p$ restricts to a group homomorphism $\overline{i_p}:H\to \bZ/2\bZ$. There are two cases: 
\begin{enumerate}
    \item For any $w_d\in H$, $p\nmid d$. In this case, $\ker{\overline{i_p}}=H$;
    \item There exists some $\tau:=w_{d_0}\in H$ such that $p\mid d_0$. In this case $K:=\ker{\overline{i_p}}$ is an index-$2$ subgroup of $H$, and $H=K\sqcup \tau K$
\end{enumerate}

Applying Proposition~\ref{prop:sqfr} inductively, we obtain the following generator-free version.

\begin{theorem}
\label{thm:generator-free}
    Let $N=p_1\dots p_r$ be a square-free positive integer, let $p=p_r$, and let $H$ be a subgroup of $W_N$. Let $\cD_H:=\{d : w_d\in H\}$.
    \begin{enumerate}
        \item If $p\nmid d$ for any $d\in \cD_H$, then $X_0(N)_{\overline{\bF}_p}/H$ obtained from taking the quotient of the Deligne--Rapoport model of $X_0(N)_{\overline{\bF}_p}$ is two irreducible components $X_0(N/p)_{\overline{\bF}_p}/H$ intersecting at 
        $$\frac{1}{|H|}\sum_{d\in \cD_H} F(d)$$
        ordinary double points. Consequently, $$g(X_0(N)/H)=2g(X_0(N/p)/H)+\left(\frac{1}{|H|}\sum_{d\in \cD_H} F(d)\right)-1.$$

        \item Otherwise, there exists some $w_{d_0}\in H$ such that $p\mid d_0$ and $K:=\ker{\overline{i_p}}$ is an index-$2$ subgroup of $H$ such that $p\nmid d$ for any $w_d\in K$. In this case, $X_0(N)_{\overline{\bF}_p}/H=(X_0(N)_{\overline{\bF}_p}/K)/\langle w_{d_0} \rangle$ obtained from taking the quotient of the Deligne--Rapoport model of $X_0(N)_{\overline{\bF}_p}$ is one irreducible component $X_0(N/p)_{\overline{\bF}_p}/K$ with
        $$\frac{1}{|H|}\sum_{d\in \cD_K} \left(F(d)-F(d*d_0)\right)$$
        self-crossings that are ordinary double points, and consequently, $$g(X_0(N)/H)=g(X_0(N/p)/K)+\frac{1}{|H|}\sum_{d\in \cD_K} \left(F(d)-F(d*d_0)\right).$$
    \end{enumerate}
\end{theorem}

\begin{remark}
\label{rmk:comparison-RH-and-trace} While we obtain our genus formulae for $g(X_0(N)/H)$ from the characteristic~$p$ perspective, one may also obtain $g(X_0(N)/H)$ formulae from the characteristic~$0$ perspective.

    From the Reimann--Hurwitz formula applied to the map $X_0(N)\to X_0(N)/H$ over $\bC$, we obtain $$(2g(X_0(N))-1)=|H|(2g(X_0(N)/H)-1)+\frac{1}{2}\sum_{d\in\cD_H\backslash\{1\}} h_N(d),$$ where $h_N(d)$ denotes the number of $w_d$-fixed points on $X_0(N)$, whose formula is given in \cite[Page 454]{Ogg} as discussed in Remark~\ref{rmk:Ogg}.  Therefore, $$g(X_0(N)/H)=\frac{1}{|H|}(g(X_0(N))-1)-\frac{1}{2|H|}\left(\sum_{d\in\cD_H\backslash\{1\}} h_N(d)\right)+1.$$

On the other hand, we can also describe $g(X_0(N)/H)$ in terms of the dimension of the $H$-fixed subspace $S_2(\Gamma_0(N))^H$ in the space of cusp forms $S_2(\Gamma_0(N))$. More precisely,
    $$g(X_0(N)/H)=\dim S_2(\Gamma_0(N))^H=\frac{1}{|H|}\sum_{d\in\cD_H}\mathrm{tr} (w_d \mid S_2(\Gamma_0(N))),$$
    where the trace of $w_d$ on $S_2(\Gamma_0(N))$ is given in \cite[Equation~(1.6)]{Martin}, also involving class numbers and Legendre symbols as our $F(D,m,p)$ formulae.
\end{remark}

As a special case of Theorem~\ref{thm:generator-free}, we have the following result for $X^*_0(N)$.
\begin{corollary}
    Let $N=p_1\dots p_r$ be a square-free positive integer and let $p=p_r$. Then $X_0^*(N)_{\overline{\bF}_p}$ obtained from taking the full Atkin--Lehner quotient of the Deligne--Rapoport model of $X_0(N)_{\overline{\bF}_p}$ is the curve $X_0^*(N/p)_{\overline{\bF}_p}$ with
$$\frac{1}{2^{r}}\sum_{d\in \cD_{r-1}}(F(d)-F(pd))$$
self-crossings that are ordinary double points, where $\cD_{r-1}$ is the set consisting of all elements of $\{1, p_1,\dotsc, p_{r-1}\}$ and all their $*$-products.

Consequently,
$$g(X_0^*(N))=g(X_0^*(N/p))+\frac{1}{2^{r}}\sum_{d\in \cD_{r-1}}(F(d)-F(pd)).$$
\end{corollary}

\begin{remark}
    \label{rmk:comparison-PPV} From our formula, we obtain $X^*_0(N)$ has good reduction over $\overline{\bF}_p$ if and only if $C:=g(X^*_0(N))-g(X^*_0(N/p))=0$. From the modular decomposition of $H^0(X_0^*(N),\Omega^1)$ in \cite[Equation~(3.5)]{Padurariu--Park--Voight} applied to modular curves ($D=1$ in their notation), we obtain $$C=\sum_{p\mid M\mid N}\dim S_2^{\text{new},\epsilon_M}(\Gamma_0(M)).$$ Therefore, our good reduction criterion agrees with theirs following from \cite[Proposition 3.3 and Lemma 3.7]{Padurariu--Park--Voight}.
\end{remark}

\newpage

\bibliographystyle{alpha}
\bibliography{references}
\end{document}